\documentclass{amsart}

\usepackage{xargs}                      
\usepackage[pdftex,dvipsnames]{xcolor}  
\usepackage{todonotes}
\newcommandx{\inbar}[2][1=]{\todo[linecolor=ForestGreen,backgroundcolor=ForestGreen!25,bordercolor=ForestGreen,#1]{#2}}
\newcommandx{\sarah}[2][1=]{\todo[linecolor=Violet,backgroundcolor=Violet!25,bordercolor=Violet,#1]{#2}}

\usepackage{amsmath, amsthm, xypic, amssymb, enumerate, mathrsfs, eufrak, verbatim, graphicx, bbm, stackrel,MnSymbol,caption,subcaption}
\usepackage{color}
\usepackage[toc,page]{appendix}
\usepackage[mathscr]{euscript}
\usepackage[all,cmtip,2cell]{xy}
\xyoption{rotate}
\UseTwocells

\usepackage[all]{xy}
\usepackage{hyperref}
\usepackage{tikz}
\usepackage{tikz-3dplot}
\usetikzlibrary{decorations.pathreplacing}

\DeclareMathOperator{\col}{colim}

\DeclareMathOperator{\Hom}{Hom}

\DeclareMathOperator{\tub}{tub}

\newcommand{\cC}{\mathcal{C}}
\newcommand{\sD}{\mathcal{D}}
\newcommand{\cE}{\mathcal{E}}
\newcommand{\cF}{\mathcal{F}}

\newcommand{\cI}{\mathcal{I}}
\newcommand{\cJ}{\mathcal{J}}

\newcommand{\cP}{\mathcal{P}}

\newcommand{\cX}{\mathcal{X}}
\newcommand{\cY}{\mathcal{Y}}

\newcommand{\cW}{\mathcal{W}}
\newcommand{\cZ}{\mathcal{Z}}

\newtheorem{thm}{Theorem}[section]

\newtheorem*{thm*}{Theorem}
\newtheorem{prop}[thm]{Proposition}
\newtheorem{cor}[thm]{Corollary}
\newtheorem{lemma}[thm]{Lemma}

\theoremstyle{definition}
\newtheorem{defn}[thm]{Definition}

\theoremstyle{definition}
\newtheorem{ex}[thm]{Example}

\theoremstyle{definition}
\newtheorem{rem}[thm]{Remark}

\theoremstyle{definition}

\theoremstyle{definition}

\newtheorem*{rem*}{Remark}
\newtheorem*{exA*}{Example}

\newtheoremstyle{TheoremForIntro} 
        {.6em}{.6em}              
        {\itshape}                      
        {}                              
        {\bfseries}                     
        {. }                             
        { }                             
        {\thmname{#1}\thmnote{ \bfseries #3}}
    \theoremstyle{TheoremForIntro}
    \newtheorem{TheoremIntro}[thm]{Theorem}
    \newtheorem{PropositionIntro}[thm]{Proposition}

\input xy
\xyoption{all}

\newcommand{\itop}{\mathsf{isvt\text{-}Top}}
\newcommand{\etop}{\mathsf{eqvt\text{-}Top}}
\newcommand{\topp}{\mathsf{Top}}
\newcommand{\imap}{\mathsf{Map_{isvt}}}
\newcommand{\emap}{\mathsf{Map_{eqvt}}}
\newcommand{\map}{\mathsf{Map}}

\newcommand{\icell}{\mathcal{I}\text{-}\mathsf{cell}}
\newcommand{\jcell}{\mathcal{J}\text{-}\mathsf{cell}}
\newcommand{\icof}{\mathcal{I}\text{-}\mathsf{cof}}
\newcommand{\jcof}{\mathcal{J}\text{-}\mathsf{cof}}
\newcommand{\iinj}{\mathcal{I}\text{-}\mathsf{inj}}
\newcommand{\jinj}{\mathcal{J}\text{-}\mathsf{inj}}
\newcommand{\horseshoe}{\mathfrak{C}}
\newcommand{\ep}{\epsilon}
\newcommand{\corner}[2]{\mathfrak{C}_{#1,#2}}
\newcommand{\iclass}[2]{[#1,#2]_{\mathsf{isvt}}}
\newcommand{\ipath}{P_{\mathsf{isvt}}}

\title[Isovariant homotopy theory and fixed point invariants]{Isovariant homotopy theory and fixed point invariants}

\author{Inbar Klang and Sarah Yeakel}

\begin{document}

\begin{abstract}
     An isovariant map is an equivariant map between $G$-spaces which strictly preserves isotropy groups. We consider an isovariant analogue of Klein--Williams equivariant intersection theory for a finite group $G$. We prove that under certain reasonable dimension and codimension conditions on $H$-fixed subspaces (for $H\leq G$), the fixed points of a self-map of a compact smooth $G$-manifold can be removed isovariantly if and only if the equivariant Reidemeister trace of the map vanishes. In doing so, we study isovariant maps between manifolds up to isovariant homotopy, yielding an isovariant Whitehead's theorem. In addition, we speculate on the role of isovariant homotopy theory in distinguishing manifolds up to homeomorphism.
\end{abstract}

\maketitle
\thispagestyle{empty}
\section{Introduction}

Advances in equivariant homotopy theory have provided important computational tools for solving problems in the nonequivariant setting. Isovariant maps, which satisfy a stronger condition than equivariance, naturally occur in surgery theory and classification questions for manifolds \cite{browderquinn}. If $X$ and $Y$ are compactly generated spaces with continuous left $G$-actions, an \emph{isovariant} map is an equivariant continuous function $f:X \to Y$ such that $G_x = G_{f(x)}$ for all $x \in X$. That is, an isovariant map preserves the $G$-action and strictly preserves isotropy subgroups.

  \medskip

  For example, consider the cyclic group with two elements, $C_2$. The one-point space $\ast$ has a trivial $C_2$-action and the unit disk $D^2$ can be given the $C_2$-action which reflects across the vertical axis. Any map from $\ast$ to $D^2$ whose image lies in the $C_2$-fixed points is both equivariant and isovariant. Indeed, any equivariant map which is injective is also isovariant. The map $D^2 \to \ast$ is equivariant, but not isovariant. While the disk with reflection action is equivariantly homotopy equivalent to the trivial one-point space, it is not isovariantly homotopy equivalent to it.


\medskip

A compelling reason to study the category of $G$-spaces with isovariant maps is the expectation that having isovariant analogues of homotopical equivariant results will provide stronger techniques for answering questions requiring more rigidity. One such question involves distinguishing manifolds up to homeomorphism. Let $M$ and $N$ be smooth, compact manifolds. It has been conjectured (see, e.g., \cite{LS}) that if the ordered configuration spaces $\mathsf{Conf}_n(M)$ and $\mathsf{Conf}_n(N)$ are homotopy equivalent for every $n$, then $M$ and $N$ are homeomorphic. For example, the lens spaces $L(7,1)$ and $L(7,2)$ are distinguished by the homotopy types of their configuration spaces \cite{LS}. Recent work on rational models of configuration spaces in \cite{CIW} suggests that this may not be the case for compact, smooth manifolds in general; the rational homotopy type of configuration spaces depends only on the rational homotopy type of the manifold in the simply-connected case, and on additional data of a partition function otherwise. 

\medskip

Instead, we propose studying the isovariant homotopy types of $M^n$ and $N^n$. The cartesian product $M^n$ has an action of the permutation group $\Sigma_n$ given by permuting the factors, and the subspace on which $\Sigma_n$ acts freely is $\mathsf{Conf}_n(M)$. Thus, if $M^n$ and $N^n$ are isovariantly homotopy equivalent for every $n$, then $\mathsf{Conf}_n(M)$ and $\mathsf{Conf}_n(N)$ are homotopy equivalent for every $n$. The isovariant homotopy type of $M^n$ contains information about the configuration spaces $\mathsf{Conf}_i(M)$, but also includes some coherence information about how they are glued to form $M^n$. Thus, a more plausible form (suggested by Cary Malkiewich) of the conjecture above might be

\medskip

\noindent \textbf{Conjecture.} If $M^n$ and $N^n$ are $\Sigma_n$-isovariantly homotopy equivalent for every $n$, then $M$ and $N$ are homeomorphic.

\medskip

This might also have interesting connections with embedding calculus, as in \cite{GKW}.

\medskip

To investigate in this direction, we study homotopy theory in the category of $G$-spaces with isovariant maps, providing new cell structures with respect to which smooth manifolds are homotopically well-behaved. As applications, we prove an isovariant Whitehead's theorem and study isovariant fixed point theory.

\medskip



A common approach in the isovariant setting is to require extra assumptions on the dimensions of gaps between isotropy subspaces so that isovariant and equivariant homotopy equivalences coincide. The gap hypotheses in this paper are significantly weaker than those in \cite{schultzgap}, and additionally weaker than those in \cite{Luc} (e.g., Definition 4.49 of that book.) We show that such strong gap hypotheses are not necessary to establish results in isovariant fixed point theory.
  
\medskip
  
Fixed point theory begins with a very classical question: given a continuous self-map of a compact space $f: X \to X$, is $f$ homotopic to a map without fixed points (elements $x \in X$ such that $f(x)=x$)? For example, the identity function of $S^1$ is homotopic to the rotation by $\pi$ map, which has no fixed points. Lefschetz defined the Lefschetz number of $f$, 
$$L(f) = \Sigma_{i=0} ^\infty (-1)^i \mathsf{rk}H_i(f)$$
and showed that for simplicial complexes $X$, if $f$ is homotopic to a map without fixed points, then $L(f) = 0$. If $X$ is a simply connected, smooth manifold of dimension at least 3, then the converse also holds: if $L(f) = 0$, then $f$ is homotopic to a map without fixed points.

\medskip

If $f: M \to M$ is a self-map of a compact, smooth manifold which is not simply connected, then $L(f)$ is not a complete invariant for the fixed point problem. In \cite{Reidemeister}, Reidemeister defined the Reidemeister trace, $R(f)$, which Wecken later showed is a complete invariant if the dimension of $M$ is at least 3 (\cite{Wecken}). That is, $R(f) = 0$ if and only if $f$ is homotopic to a map without fixed points (see also \cite{Hus}.) The Reidemeister trace $R(f)$ can be considered as an element of $H_0(L_f M)$, where $L_f M = \{ \gamma: I \to M : \gamma(1) = f(\gamma(0)) \}$ is the twisted free loop space; see, e.g., section 10 of \cite{KW1} or chapter 6 of \cite{PonThesis}.

\medskip

The fixed point problem for equivariant maps has proven to be more complicated. The equivariant Lefschetz trace, $L_G(f)$, of a $G$-equivariant map $f: M \to M$ can be considered as an element of the Burnside ring of $G$ (see, e.g., Remark 8.5.2 of \cite{tD}.) In the case of the identity map, $L_G(id)$ recovers the equivariant Euler characteristic of $M$. A complete invariant in this case is given by the equivariant Reidemeister trace, $R_G(f)$, which can be considered as an element of $H_0 ^G (L_f M)$. Showing that $R_G(f)$ is a complete invariant is substantially more difficult than doing the same for $R(f)$, in large part due to issues of equivariant transversality. Results of this flavor are proven in \cite{FW}, \cite{won}, \cite{Fer}, \cite{Web}, and \cite{KW2}. In \cite{KW1} and \cite{KW2}, Klein and Williams develop a homotopy theoretic approach to intersection theory, and use it to address the equivariant fixed point problem using equivariant homotopy theory rather than equivariant transversality. They prove:

\medskip

\noindent \textbf{Theorem H.} \cite{KW2} Let $G$ be a finite group, and let $M$ be a compact smooth $G$-manifold such that 
\begin{itemize}
\item $\dim M^H \geq 3$ for all $H$ that appear as isotropy groups in $M$, and
\item for $H < K$ that appear as isotropy groups in $M$, $\dim M^H - \dim M^K \geq 2$.
\end{itemize}
Let $f: M \to M$ be an equivariant map. Then $f$ is equivariantly homotopic to a map without fixed points if and only if $R_G(f)$ is trivial.

\medskip

Here, $M^H$ denotes the points in $M$ fixed by the subgroup $H$.

\medskip

In this paper, we address the fixed point problem for isovariant maps:

\medskip

\noindent \textbf{Question.} Given a compact, smooth $G$-manifold $M$ and an isovariant map $f: M \to M$, when can we find an isovariant map $f_1: M \to M$, isovariantly homotopic to $f$, such that $f_1$ has no fixed points?

\medskip

Motivated by the impressive effectiveness of equivariant homotopy theory in \cite{KW2}, we further the study of isovariant homotopy theory by providing analogues of certain important equivariant results, described below. This allows us to prove:

\medskip
\begin{TheoremIntro}[\ref{thm-finite-G}] Let $G$ be a finite group and suppose that $M$ is a compact smooth $G$-manifold such that
\begin{itemize}
\item $\dim M^H \geq 3$ for all $H$ that appear as isotropy groups in $M$, and
\item For $H < K$ that appear as isotropy groups in $M$, $\dim M^H - \dim M^K \geq 2$
\end{itemize}
Let $f: M \to M$ be an isovariant map. If the fixed points of $f$ can be removed equivariantly (equivalently, if $R_G(f) = 0$), then they can be removed isovariantly.
\end{TheoremIntro}

\medskip

That is, we obtain the result that equivariant fixed point theory and isovariant fixed point theory for such manifolds coincide. The equivariant Reidemeister trace $R_G(f)$ completely determines whether an isovariant self-map of such manifolds can be modified by an isovariant homotopy to a fixed point free map.

\medskip

To obtain the theorem above, we use the homotopical techniques of Klein--Williams in the isovariant setting, which requires the presence of a homotopy theory for the isovariant category. In \cite{sayisvtelm}, the second author constructed a Quillen model structure on the category of $G$-spaces with isovariant maps. While this model structure provides a way to combinatorially represent $G$-spaces as presheaves on a category, we cannot apply the techniques of \cite{KW2} to $G$-manifolds with this homotopy theory, because not all $G$-manifolds are cofibrant. In this paper, we develop a more robust homotopy theory of the isovariant category (for $G$ a finite group), which involves defining new kinds of isovariant cells (see Definition \ref{defn:geometricij}.) With these new isovariant cell structures, we prove

\medskip

\begin{TheoremIntro}[\ref{thm-mflds-cofibrant} and \ref{thm-mflds-fibrant}]
 Let $G$ be a finite group. Smooth $G$-manifolds are built out of isovariant cells, and
   they satisfy lifting properties with respect to isovariant cell inclusions.
\end{TheoremIntro}

Theorem \ref{thm-mflds-fibrant} is closely related to the fact that conically stratified spaces are fibrant, a result of \cite[4.12]{douteaustrat} that relies on Theorem A.6.4 of \cite{lurieHA}. We carry out a similar argument in the equivariant setting, considering smooth $G$-manifolds instead of the more general notion of conically stratified spaces. More work on conically stratified spaces can be found, for example, in \cite{AFT}. 

\medskip

In the presence of a model structure, the theorems above would mean that manifolds are cofibrant and fibrant. While we do not provide a new Quillen model structure on isovariant spaces, we are able to obtain the main results that would follow from a suitable one. These results can be interpreted as arising from a weak model structure in the sense of \cite{Hen}.

\begin{TheoremIntro}
    [\ref{weakmodelstr}] Let $G$ be a finite group. There is a weak model structure on $\itop$ in which smooth $G$-manifolds are cofibrant and fibrant.
\end{TheoremIntro}

Proposition A.1 of Douteau--Waas \cite{DW} indicates that the Quillen model structure on isovariant spaces that we nearly construct in this paper would need to rely on more than the stratification by isotropy. While their counterexample cannot arise in the isovariant setting, we are currently unable to prove that $\jcell \subseteq \mathcal{W}$ (where $\jcell$ and $\mathcal{W}$ are defined in section 3.)  

\medskip

As a result of our study of this more robust isovariant homotopy theory, we are able to prove an isovariant Whitehead's theorem.

\begin{TheoremIntro}[\ref{thm-whitehead}](Isovariant Whitehead's theorem)
Let $G$ be a finite group, let $X$ and $Y$ be smooth $G$-manifolds, and let $f:X \to Y$ be an isovariant map. Then $f$ is an isovariant homotopy equivalence if and only if $f$ is an isovariant weak equivalence (in the sense of Definition \ref{defn:elemIJ}).
\end{TheoremIntro}

\medskip

In \cite{dulaschultz}, an isovariant version of Whitehead's theorem was proven for manifolds with ``treelike isotropy structure''; we show that Theorem \ref{thm-whitehead} extends and strengthens Theorem 4.10 of \cite{dulaschultz}. That is, we show

\medskip
\begin{PropositionIntro}[\ref{prop-equiv-DS}]
Let $G$ be a finite group and let $f:X \to Y$ be an isovariant map between smooth $G$-manifolds with treelike isotropy structure satisfying the conditions of Theorem 4.10 in \cite{dulaschultz}. Then $f$ is an isovariant weak equivalence.
\end{PropositionIntro}

\medskip

Further, we characterize isovariant weak equivalences between $G$-manifolds in terms of simpler data. That is, instead of checking that $f$ induces a weak equivalence on isovariant mapping spaces of infinitely many higher dimensional simplices, we show that $f$ is an isovariant weak equivalence of $G$-manifolds if it induces a weak equivalence on the isovariant mapping spaces of zero and one dimensional linking simplices.

\subsection*{Structure of the paper.} In section 2, we cover necessary preliminaries, such as background on the category of isovariant spaces and on homotopical fixed point theory. In section 3, we define classes of maps out of which isovariant cell complexes are built, and which govern the homotopy extension and lifting properties we are interested in. We then show that all smooth $G$-manifolds are isovariant cell complexes and satisfy lifting properties with respect to isovariant cell inclusions. We use this to prove the isovariant Whitehead's theorem. In section 4, we apply our results on isovariant homotopy theory to isovariant fixed point theory; we prove that for smooth $G$-manifolds with assumptions as in \cite{KW2}, the isovariant fixed point problem reduces to the equivariant fixed point problem. We also give a counterexample showing that this is not necessarily the case for $G$-spaces which are not manifolds. In section 5, we provide a weak model structure on $\itop$ in which smooth $G$-manifolds are cofibrant and fibrant.

\subsection*{Acknowledgments.} This project was inspired by conversations at a Women in Topology workshop at the Mathematical Sciences Research Institute. We would like to thank Kate Ponto and Cary Malkiewich for suggesting the project and providing support throughout. We would also like to thank Michael Mandell for an illuminating discussion on Theorem \ref{thm-whitehead}, and the anonymous referee whose comments greatly improved the quality of this paper. We would like to thank Peter Haine for introducing us to the notion of weak model categories, and Maru Sarazola for an enlightening talk and useful conversations about fibrantly induced model structures. Some of this paper is based upon work supported by the National Science Foundation under Grant No. 1440140, while the first-named author and the second-named author were in residence at the Mathematical Sciences Research Institute in Berkeley, California, during the semester of Fall 2022 and the Spring 2020 semester, respectively.

\section{Preliminaries}

We assume familiarity with the basics of cofibrantly generated model structures as presented in \cite{hovey}. In particular, we use the following notation for classes of maps with certain lifting properties. Let $\mathcal{I}$ be a class of maps in a category $\mathcal{C}$ containing all small colimits.

$\iinj$ is the class of maps with the right lifting property with respect to every map in $\mathcal{I}$.

$\icof$ is the class of maps with the left lifting property with respect to all maps in $\iinj$. 

$\icell$ is the class of relative $\mathcal{I}$-cell complexes. A relative $\mathcal{I}$-cell complex is a transfinite composition of pushouts of coproducts of elements of $\mathcal{I}$. Note that $\icell \subset \icof$, and in fact, $\icof$ consists of retracts of maps in $\icell$.

\subsection{Isovariance}

Let $G$ be a finite group. We denote by $\etop$ the category of $G$-spaces (compactly generated  spaces with continuous left $G$-action) with equivariant maps, and we denote by $\itop$ the category of $G$-spaces with isovariant maps with an added formal terminal object. (We note that in \cite{sayisvtelm}, this category is denoted $\itop^\triangleright$.)
The category $\etop$ is enriched in spaces, and thus $\itop$ is also enriched in spaces using the subspace topology. We will denote the space of equivariant maps from $X$ to $Y$ by $\emap(X,Y)$ and the space of isovariant maps by $\imap(X,Y)$.

\medskip

Let $s_n:S^n \to D^{n+1}$ be the usual boundary inclusion and denote $i_0:\{0\} \to [0,1]$. 
There is a cofibrantly generated model structure on the category $\topp$ of compactly generated spaces where the generating cofibrations $\cI^\topp$ are given by $\{s_{n}:S^{n} \to D^{n+1}\}$ and generating acyclic cofibrations $\cJ^\topp$ are given by $\{D^n \times i_0:D^n \times \{0\} \to D^n \times [0,1]\}$ \cite[2.4]{hovey}. The family of adjunctions below, where $H$ runs over all subgroups of $G$, allows this model structure to transfer to $\etop$ \cite{stephanelm}. 
\[
\left\{G/H \times -:\xymatrix{\mathsf{Top} \ar@<1ex>[r] & \etop \ar@<1ex>[l] } : \emap(G/H, -) \right \}_{H \leq G}
\]
In particular, the cofibrantly generated model structure on $\etop$ has generating cofibrations given by $\{s_n \times G/H\}$, generating acyclic cofibrations given by $\{D^n \times i_0 \times G/H\}$, and $f:X \to Y$ is a weak equivalence if and only if $\emap(G/H,f)$ is a weak equivalence of spaces for all subgroups $H \leq G$.

\medskip

In \cite{sayisvtelm}, similar techniques are used to transfer the model structure of $\topp$ to $\itop$. We will refer to this model structure as the \emph{elementary model structure}. To define the generating (acyclic) cofibrations and weak equivalences, we need to introduce linking simplices. Because isovariant maps between $G$-spaces preserve the stratification by isotropy groups, we will consider chains of subgroups of $G$. Let $\mathbf{H}=H_0< \cdots<H_n$ be a strictly increasing chain of subgroups of $G$. 

\medskip

Let $\Delta^n$ be the standard $n$-simplex in $\topp$, that is,
\[
\Delta^n = \left\{(t_0, \dots, t_n) \in [0,1]^{n+1} : \sum_{i=0}^n t_i =1 \right\}.
\]

\begin{defn} \label{defn:linkingsimplex}
  The \emph{linking simplex} $\Delta_G^{\mathbf{H}}$ is the quotient of $G \times \Delta^n$ where $(g,x)\sim(g',x)$ if and only if $gH_{k}=g'H_{k}$, when $x=(t_0, \dots, t_{n-k}, 0, \dots, 0)$, $0 \leq k \leq n$. Let $G \times \Delta^n \to \Delta^{H_0< \cdots< H_n}_G$ be the natural projection and denote the image of $(g,x)\in G \times \Delta^n$ under the projection by $\langle g,x\rangle$. The space $\Delta^{\mathbf{H}}_G$ has a left $G$-action given by $g' \cdot \langle g,x \rangle=\langle g'g,x \rangle$; points of the form $\langle g,(t_0, \dots, t_{n-k},0, \dots,0)\rangle$ where $t_{n-k}\neq 0$ are fixed by $gH_kg^{-1}$ under the $G$-action. 
\end{defn}

\begin{ex}\label{ex-G/H}
    If $\mathbf{H} = \{ H_0 \}$, then $\Delta_G^{\mathbf{H}} = G/H_0$. This will also be denoted $\Delta^{H_0}$.
\end{ex}


We note that $\Delta^{H_0< \cdots < H_n}_G$ is the same as the ``equivariant simplex'' $\Delta_k(G;H_n, \dots, H_0)$ defined in \cite{illmanCW}, although in the equivariant simplex, subgroups may be repeated. Illman shows that the equivariant simplex is a compact Hausdorff space with orbit space $\Delta^n$. We will often consider a fundamental domain of a linking simplex, denoted $fd(\Delta^\mathbf{H})$. When $G$ is clear from context, we drop it from the notation.

\medskip

The boundary of a linking simplex, $\partial \Delta^{\mathbf{H}}_G$, is the image of $\partial \Delta^n \times G$ (in the usual sense) under the identifications. This can also be identified with 
\[
\partial \Delta^{\mathbf{H}}_G= \col_{\bullet < \mathbf{H}} \Delta^{\bullet}_G,
\]
where $\bullet < \mathbf{H}$ denotes all proper subchains of $\mathbf{H}$. 
Denote the boundary inclusion of a linking simplex by $b^\mathbf{H}: \partial \Delta^\mathbf{H} \to \Delta^\mathbf{H}$.

\medskip

Let $H \leq G$ be a subgroup. In the equivariant setting, $\emap(G/H,X)$ is equivalent to $X^H$, the subspace of elements which are $H$-fixed. In the isovariant setting, $\imap(\Delta^H,X)$ picks out the subspace of $X$ fixed by exactly $H$, denoted $X_H$.

\medskip

\begin{defn} \label{defn:elemIJ}
The elementary model structure for isovariant spaces has \emph{generating cofibrations} $\mathcal{I}^{elem}$ given by $\{s_n \times \Delta^\mathbf{H}:S^n \times \Delta^\mathbf{H} \to D^{n+1} \times \Delta^\mathbf{H}\}$ and \emph{generating acyclic cofibrations} $\mathcal{J}^{elem}$ given by $\{D^{n+1} \times \Delta^\mathbf{H} \times i_0\}$ (or equivalently $(s_n \Box i_0) \times \Delta^\mathbf{H}$). The \emph{weak equivalences} of $\itop$ are the isovariant maps $f:X \to Y$ such that $\imap(\Delta^\mathbf{H},X) \to \imap(\Delta^\mathbf{H},Y)$ are weak equivalences of spaces for all strictly increasing chains $\mathbf{H}=H_0< \cdots< H_n$. We call these maps \textbf{isovariant weak equivalences}. 
\end{defn}

The formal terminal object $T$ is isovariantly contractible. Thus a map $f:X \to T$ is a weak equivalence if $\imap(\Delta^\mathbf{H},f)$ is a weak equivalence of spaces for each $\mathbf{H}$. The elementary model structure is Quillen equivalent to the category of presheaves on the link orbit category with the projective model structure. For more details, see \cite{sayisvtelm}.

\begin{rem}\label{rem-homotopy-equiv}
An isovariant homotopy equivalence $f$ is in particular an isovariant weak equivalence, since $\imap(\Delta^\mathbf{H},f)$ is a homotopy equivalence for all $\mathbf{H}$.
\end{rem}

One drawback of the elementary model structure on $\itop$ is that not all $G$-manifolds are cofibrant. For example, the 2-sphere with a rotation action (sometimes denoted $S^\lambda$) is not cofibrant.  In section \ref{sec-geom-model-str}, we develop more robust homotopy theoretic tools for the isovariant category, with which $G$-manifolds are particularly well-behaved. In order to define the relevant classes of maps in section \ref{sec-geom-model-str}, we use the pushout-product of maps in $\topp$ with maps in $\itop$. We rely heavily on the relationship between pushout-products and pullback-homs described below.

\medskip

Let $\cC, \sD,$ and $\cE$ be categories with all small colimits, and let $\otimes:\cC \times \sD \to \cE$ be a colimit-preserving functor. Then the pushout-product of a map $f:K \to X$ in $\cC$ and $g:L \to Y$ in $\sD$ is the map $f \Box g$ in $\cE$ from the pushout of the first three terms in the following square to the final vertex.
\[
\xymatrix{
  K \otimes L \ar[r]^{id \otimes g} \ar[d]_{f \otimes id} & K \otimes Y \ar[d]^{f \otimes id} \\
  X \otimes L \ar[r]^{id \otimes g} & X \otimes Y
}
\]
If $\cC$ also has all small limits and there is a functor $\Hom_\sD:\sD^{op} \times \cE \to \cC$ with an adjunction between $-\otimes d$ and $\Hom_\sD(d,-)$ for every object $d \in \sD$, then the pullback-hom $\Hom_{\Box}(g,h)$ of $g:L \to Y$ in $\sD$ and $h:M \to Z$ in $\cE$ is the map in $\cC$ from the initial vertex to the pullback of the last three vertices of the square below.
\[
\xymatrix{
  \Hom(Y,M) \ar[r] \ar[d] & \Hom(L,M) \ar[d] \\
  \Hom(Y,Z) \ar[r] & \Hom(L,Z)
}
\]

An adjunction between pushout-products and pullback-homs is described in \cite[3.2.3]{caryparam}, which yields the following useful relationship between lifts. 

\begin{lemma} \label{lemma:popr-adj-pbhom} For maps as described above, a lift exists in the first diagram below if and only if a lift exists in the second diagram \cite[19.5]{rezkquasi}.
\[
\xymatrix{
  X \otimes L \coprod_{K \otimes L} K \otimes Y \ar[d]_{f \Box g} \ar[r] & M \ar[d]^h \\
  X \otimes Y \ar[r] \ar@{-->}[ur] & Z
}
 \, \, \, \, \, \, \, \ \ 
\xymatrix{
  K \ar[r] \ar[d]^f & \Hom(Y,M) \ar[d]^{\Hom_\Box(g,h)} \\
  X \ar[r] \ar@{-->}[ur] & \Hom(L,M)\times_{\Hom(L,Z)}\Hom(Y,Z)
}
\]
\end{lemma}

\subsection{Homotopical fixed point theory} \label{subsec-fixedpoints}

In \cite{KW1} and \cite{KW2}, Klein and Williams use a homotopy theoretic approach to define fixed point invariants in a way that does not rely on transversality. In this subsection, we review the aspects of their approach that will appear in this paper. Given an equivariant self-map of a compact smooth $G$-manifold $f: M \to M$, the problem of finding an equivariantly homotopic map $f_1: M \to M$ such that $f_1$ has no fixed points reduces to the problem of finding a lift up to equivariant homotopy of
$$id \times f: M \to M \times M$$
to $M \times M - \Delta$, where $\Delta$ denotes the diagonal. Replacing the inclusion map $M \times M - \Delta \to M \times M$ with an equivariant fibration $E' \to M \times M$ and pulling this back along the map $id \times f: M \to M \times M$ yields an equivariant fibration $p : E \to M$. Note that the problem of finding a map $f_1$ as above then transforms into finding a section of the equivariant fibration $p: E \to M$.

\medskip

Denote by $S_M E$ the unreduced fiberwise suspension of $E$ over $M$. That is, 
$$S_M E = M \cup_{E \times 0} (E \times I) \cup_{E \times 1} M.$$
We can consider $S_M E$ as a retractive $G$-space over $M$, with section $M \to S_M E$ given by inclusion into the 0-copy of $M$. The space $M \amalg M$ is a retractive $G$-space over $M$, with section $M \to M \amalg M$ given by the inclusion as the left summand. Then
$$s: M \amalg M \to S_M E$$
given by inclusion of the 0-copy and the 1-copy defines a map of retractive $G$-spaces over $M$.
One of the central ideas of \cite{KW1} and \cite{KW2} converts the problem of finding a section of $p: E \to M$ into the problem of determining whether $s$ is trivial. The following is obtained by combining Proposition 3.1 with Lemmas 10.1 and 10.2 of \cite{KW2}.

\begin{thm}\label{thm-equivt-euler}
Suppose that $M$ is a compact smooth $G$-manifold such that
\begin{itemize}
\item $\dim M^H \geq 3$ for all $H$ that appear as isotropy groups in $M$, and
\item for $H < K$ that appear as isotropy groups in $M$, $\dim M^H - \dim M^K \geq 2$.
\end{itemize}
Then the fibration $p: E \to M$ has an equivariant section if and only if the homotopy class of $s: M \amalg M \to S_M E$ in retractive $G$-spaces over $M$,
$$[s] \in [M \amalg M, S_M E]^G_M ,$$
is trivial. 
\end{thm}

Thus, under the above conditions, $[s]$ provides a complete invariant for equivariantly eliminating fixed points. The class $[s] \in [M \amalg M, S_M E]^G_M$ is denoted by $R_G(f)$ and called the equivariant Reidemeister trace of $f$. 
In \cite{KW2}, $R_G(f)$ is considered as an element of $\pi_0 ^G (\Sigma^\infty _+ L_f M)$, via stabilizing to parametrized $G$-spectra and applying Poincar\'e duality. Here $L_f M = \{ \gamma: I \to M : f(\gamma(0)) = \gamma(1) \}$ denotes the twisted free loop space. In this paper, we will not need this stabilized version.

\medskip

The proof of Theorem \ref{thm-equivt-euler} in \cite{KW2} proceeds by induction on the fixed submanifolds $M^H$. Write $(H) < (K)$ if $H$ is properly subconjugate to $K$, and let
$$ (H_1) > ... >(H_n) $$
be a total ordering of the subgroups appearing as isotropy groups of elements of $M$, which extends the subconjugacy order. For $1 \leq k \leq n$, denote by $M_k = \cup_{i \leq k} M^{(H_i)}$ the subspace of $M$ consisting of all points fixed by a subgroup conjugate to some $H_i$, $i \leq k$. For example, $M_1 = M^G$ (if $M^G$ is nonempty) and $M_n = M$.

\medskip

If $f: M \to M$ is an equivariant map that has no fixed points on a $G$-invariant subcomplex $A \subseteq M$, the obstruction for removing its fixed points on a $G$-invariant subcomplex $B \subseteq M$ such that $A \subseteq B$ is a ``local Reidemeister trace" 
$$R_G(f|_{A, B}) \in [C_M (B, A), S_M E]^G _M$$
where 
$$C_M(B,A) = M \cup_{A \times 0} A \times I \cup_{A \times 1} B$$
is a homotopy cofiber of $M \amalg A \to M \amalg B$ over $M$. See Theorem D of \cite{KW2}, Proposition 8.2 and Lemma 8.4 of \cite{Pon}, or Theorem 1.6 of \cite{PonKW}. For example, condition (ii) in Theorem 1.6 of \cite{PonKW} is a Blakers--Massey condition, which is satisfied if $\dim M^H \geq 3$ for all $H$ that appear as isotropy groups in $M$, by the equivariant Blakers--Massey theorem (see, e.g., \cite{Dot}). In the inductive approach of \cite{KW2}, it is therefore crucial to show that the local Reidemeister traces $R_G(f|_{M_{k-1}, M_k})$ vanish.

\medskip

Theorem E, along with Lemmas 10.1 and 10.2 of \cite{KW2}, provides the necessary ``global-to-local" theorem:

\begin{thm}\label{thm-KW-global-to-local}
Suppose that $M$ is a compact smooth $G$-manifold such that
\begin{itemize}
\item $\dim M^H \geq 3$ for all $H$ that appear as isotropy groups in $M$, and
\item for $H < K$ that appear as isotropy groups in $M$, $\dim M^H - \dim M^K \geq 2$.
\end{itemize}
Suppose that $f: M \to M$ is an equivariant map. If $R_G(f)$ is trivial, then $R_G(f|_{M_{k-1}, M_k})$ is trivial for all $k$.
\end{thm}

This global-to-local theorem allows for removing fixed points on $M$ inductively, at each step proceeding from $M_{k-1}$ to the subsequent $M_k$. This will also be our method of isovariantly removing fixed points in Section \ref{sec-isvt-fixed-pt}.

\section{Homotopy extension and lifting for isovariant spaces} \label{sec-geom-model-str}
In order to study isovariant fixed point theory for $G$-manifolds and prove an isovariant Whitehead theorem, we will show that manifolds satisfy certain homotopy extension and lifting properties for isovariant maps. This almost amounts to a new Quillen model structure on isovariant spaces, but there are subtle technicalities in building the model structure that we are unable to resolve. We are thankful for work of Douteau--Waas \cite{DW}, whose Proposition A.1 alerted us to the stratified version of this issue. While their counterexample A.4 cannot arise in the isovariant setting, we are currently unable to prove that $\jcell \subseteq \mathcal{W}$ (where $\jcell$ and $\mathcal{W}$ are defined below.)  Nevertheless, as in work of Douteau and Douteau--Waas \cite{douteaustrat, DW} on stratified spaces, we are able to obtain the main results that would follow from a suitable model structure.

\medskip

We will define three classes of maps $\mathcal{I}$, $\mathcal{J}$ and $\mathcal{W}$. The class $\mathcal{I}$ generates the isovariant cell inclusions. In a Quillen model structure, $\mathcal{I}$ would be the class of generating cofibrations, and $\mathcal{J}$ would be the class of generating acyclic cofibrations. Recall that $s_{n}:S^n \to D^{n+1}$ and $b^\mathbf{H}:\partial\Delta^\mathbf{H}_G \to \Delta^\mathbf{H}_G$ are the boundary inclusions, and $\Delta^\mathbf{H}_G$ is the linking simplex of Definition \ref{defn:linkingsimplex}. 

\begin{defn} \label{defn:geometricij} Define $\mathcal{I}$ as the class of pushout-products of $s_{n-1}$ with $b^\mathbf{H}$, that is,

\[
\mathcal{I}=\left\{s_{n-1} \Box b^\mathbf{H}: \left(S^{n-1} \times \Delta^{\mathbf{H}}_G \right) \cup_{S^{n-1}\times \partial \Delta_G^{\mathbf{H}}} \left(D^n \times \partial \Delta^{\mathbf{H}}_G \right) \to D^n \times \Delta^{\mathbf{H}}_G\right\}_{\mathbf{H},n}
\]
Define $\mathcal{J}$ as pushout-products of the maps in $\mathcal{I}$ with $i_0:\{0\} \to [0,1]$, that is,
\[
\mathcal{J}=\left\{s_{n-1} \Box b^\mathbf{H} \Box i_0\right\}_{\mathbf{H},n}
\] 
Let $\mathcal{W}$ be the class of isovariant weak equivalences, that is, the maps $f:X \to Y$ such that the induced map $\imap(\Delta^\mathbf{H},X) \to \imap(\Delta^{\mathbf{H}},Y)$ is a weak equivalence of spaces for all strictly increasing chains of subgroups $\mathbf{H}=H_0<\cdots<H_n$.
\end{defn} 

\medskip

We will prove a series of lemmas about the classes $\mathcal{I}, \mathcal{J}$ and how they interact with the isovariant weak equivalences $\mathcal{W}$. This will enable us to prove an isovariant Whitehead's theorem, as well as results on isovariant fixed point theory in section 4. 

\medskip

We begin by showing that cofibrations in the elementary model structure are built out of maps in $\mathcal{I}$, that is, $\mathcal{I}^{elem} \subseteq \icell$, and that acyclic cofibrations in the elementary model structure are built out of maps in $\mathcal{J}$, that is, $\mathcal{J}^{elem} \subseteq \jcell$ (see Definition \ref{defn:elemIJ}). First, given a poset $\mathcal{P}$ of chains of subgroups of $G$, define $\mathcal{P} \Delta$ to be the colimit $\col_{\mathbf{H} \in \mathcal{P}} \Delta^\mathbf{H}$. We say that $\mathcal{P}$ is closed under inclusion if whenever $\mathbf{H} \in \mathcal{P}$ and $\mathbf{K}$ is a subchain of $\mathbf{H}$, then $\mathbf{K} \in \mathcal{P}$. For example, if $\mathbf{H}$ is a chain of subgroups, then the poset $\mathcal{P}_\mathbf{H}$ of all subchains of $\mathbf{H}$ is closed under inclusion, and $\mathcal{P}_\mathbf{H} \Delta = \Delta ^\mathbf{H}$.

\begin{lemma}\label{lem-eltry-cof}
\begin{enumerate}
    \item For every $\mathcal{P}$ closed under inclusion, the maps $s_n \times \mathcal{P} \Delta: S^n \times \mathcal{P} \Delta \to D^{n+1} \times \mathcal{P} \Delta$ are in $\icell$. That is, they are obtained by composition of pushouts of coproducts of elements of $\mathcal{I}$.
    \item For every $\mathcal{P}$ closed under inclusion, the maps $D^n \times i_0 \times \mathcal{P} \Delta: D^n \times 0 \times \mathcal{P} \Delta \to D^n \times [0,1] \times \mathcal{P} \Delta$ are in $\jcell$. That is, they are obtained by composition of pushouts of coproducts of elements of $\mathcal{J}$.
\end{enumerate}
\end{lemma}

In particular, $s_n \times \Delta^\mathbf{H} \in \icell$ and $D^n \times i_0 \times \Delta^\mathbf{H} \in \jcell$.

\begin{proof}
We will prove (1) by induction on the lengths of the chains in $\mathcal{P}$. If all chains in $\mathcal{P}$ are of length 0, then 
$$\mathcal{P} \Delta = \Delta^{H_0} \amalg ... \amalg \Delta^{H_m}.$$
If $\mathbf{H} =  \{ H_0 \}$ has length 0, then $s_n \times \Delta^{H_0}$ is equal to $s_n \Box b^{H_0}$, so it is in $\mathcal{I}$ for all $n$. The coproduct of maps in $\mathcal{I}$ is in $\icell$, therefore $s_n \times \mathcal{P} \Delta \in \icell$.

\medskip

Now suppose the statement is true for $\mathcal{P}$, and let $\mathbf{K}$ be a chain of subgroups which is not in $\mathcal{P}$. Define $\mathcal{P}' = \mathcal{P} \cup  \{ \mathbf{K} \}$. Note that $\mathcal{P}' \Delta$ is the pushout of the diagram

$$\xymatrix{
\col_{\mathbf{H} \in \mathcal{P}} \Delta^{ \mathbf{H} \cap \mathbf{K}} \ar[r]^-f \ar[d] & \mathcal{P} \Delta \\
\Delta^\mathbf{K}
}$$

We may assume that all proper subchains of $\mathbf{K}$ are in $\mathcal{P}$, and therefore this diagram becomes

$$\xymatrix{
\partial \Delta^\mathbf{K} \ar[r]^-f \ar[d]^{b^\mathbf{K}} & \mathcal{P} \Delta \\
\Delta^\mathbf{K}
}$$

Multiplying by $S^n$ and $D^{n+1}$, we obtain the map of pushout diagrams

$$\xymatrix{
S^n \times \Delta^\mathbf{K} \ar[d]^{s_n \times \Delta^\mathbf{K}}    & \ar[l]^-{S^n \times b^\mathbf{K}}  S^n \times \partial \Delta^\mathbf{K} \ar[r]^-{S^n \times f} \ar[d]^{s_n \times \partial \Delta^\mathbf{K}}  & S^n \times \mathcal{P} \Delta \ar[d]^{s_n \times \mathcal{P} \Delta} \\
D^{n+1} \times \Delta^\mathbf{K}    & \ar[l]^-{D^{n+1} \times b^\mathbf{K}}  D^{n+1} \times \partial \Delta^\mathbf{K} \ar[r]^-{D^{n+1} \times f}  & D^{n+1} \times \mathcal{P} \Delta 
}$$

The map $s_n \Box b^\mathbf{K}$ is in $\icell$ by definition, and $s_n \times \mathcal{P} \Delta \in \icell$ by the induction hypothesis. Therefore by Corollary 2.2.2 of \cite{caryparam}, the map between the pushouts, $s_n \times \mathcal{P}' \Delta$, is also in $\icell$, as required.

\medskip 

The proof of (2) is similar, using the fact that $D^n \times i_0$ is  $s_{n-1} \Box i_0$ composed with a homeomorphism.
\end{proof}

Next, we prove a lemma about cubical limits of maps, which will be applied in the proof of Lemma \ref{lem-I-inj} to $\cI^\topp$-$\mathsf{inj}$. Recall that an \emph{$S$-cube} in a category $\cC$ is a functor $\cX:\cP(S)\to \cC$, where $S$ is a finite set and $\cP(S)$ is the poset of all subsets of $S$. The finite set $\underline{n}=\{1, 2, \dots, n\}$ will define an \emph{$n$-cube}. For an $S$-cube $\cX$, and subsets $U \subset T \subset S$, we use $\partial_U^T\cX$ to denote the $(T-U)$-cube $\{V \mapsto \cX(V \cup U):V \subset T-U\}$. For a subset $U \subset S$, we define the \emph{$U$-corner map} of $\cX$ to be the restriction map
$$\lim_{U \subset T} \cX (T)  \to \lim_{U \subsetneqq T} \cX(T)$$
and $\lim_{U \subset T} \cX (T)=\cX (U)$.
For the following lemma, note that a map of $n$-cubes $\cX \to \cY$ defines an $(n+1)$-cube.

\begin{lemma} \label{cubelimitsfibration}
  Let $\cC$ be a category with all limits and a class of maps $\cF$ which are preserved by pullbacks and compositions. Assume also that isomorphisms are in $\cF$. Let $\cX \to \cY$ be a map of cubical diagrams in $\cC$ such that the corner maps of all possible subcube maps $\cX' \to \cY'$ are in the class $\cF$. Then $\lim \cX \to \lim \cY$ is in the class $\cF$.
\end{lemma}

We note that the corner map condition in the lemma is slightly weaker than the fibration cube condition in Definition 1.13 of \cite{G2}, since we do not require corner maps of subcubes of $\cX$ or $\cY$ to be in $\cF$. We are unaware of a proof of this lemma in the literature, so we provide one now. 

\begin{proof}
  We will factor the map $\lim \cX \to \lim \cY$ through maps of $\cF$. Let $\cX$ and $\cY$ be cubes in $\cC$ indexed by the poset $\cP (\underline{n})$ on $\underline{n}=\{1,2, \dots, n\}$. 
  
  \medskip

Let $\cX_0$ denote the final vertex $\cX(\underline{n})$ of the cube $\cX$, and let $\cX_i$ denote the diagram $\cup_{j=1}^i \partial_{\underline{n}-\{j\}}^{\underline{n}} \cX$ for $1 \leq i \leq n$. For example, $\cX_n$ is the full subdiagram of $\cX$ containing all vertices of distance at most 1 from the final vertex. We continue this process to define $\cX_i$ for all $n <i \leq 2^n-1$ by incrementally adding vertices of distance $2, 3, \dots,$ and finally $n$ from the final vertex. We apply the same vertex ordering to define $\cY_i$ for all $0 \leq i \leq 2^n-1$. Note this process extends the partial order on the poset $\cP(\underline{n})$ to a total order, although we do not add extra maps to the associated domain category of the diagrams. 

\medskip

We define $L_i$ for $0 \leq i \leq 2^n-1$ to be the limit of the diagram containing $\cX_i$ and $\cY$ and the maps between them (that is, the diagram maps $\cX_i \to \cY_i$). 
We denote this by $L_i=\lim(\cX_i \to  \cY)$. 
The limit $L_i$ can be rewritten as the pullback of the following diagram for $i\geq 1$:
   \[
  \xymatrix{ & \lim(\cX_i \to \cY_i) \ar[d]^{\theta_i} \\ L_{i-1} \ar[r] & \lim(\cX_{i-1} \to \cY_i)}
  \]
  The map $L_i \to L_{i-1}$ is in $\cF$ if the right vertical map $\theta_i$ is in $\cF$. 
  For $i=0$, $L_0$ is the pullback of the diagram
  \[
   \xymatrix{ & \cX_0 \ar[d]^{\theta_0} \\ \lim \cY \ar[r] & \cY_0}
   \]
   Thus the map $L_0 \to \lim \cY$ is in $\cF$ if $\cX_0 \to \cY_0$ is in $\cF$.
   
   \medskip
  
  We will now show that the right vertical maps $\theta_i$ of the previous diagrams are in $\cF$. First, $\theta_0$ is the corner map of the subdiagram $\cX_0 \to \cY_0$.
  
  \medskip

  Adding the $j$th vertex, $\cY(\underline{n}-U)$ (where $U$ depends on the total order chosen above), to the diagram $\cY_{j-1}$ completes the $|U|$-cube $\partial^{\underline{n}}_{\underline{n}-U} \cY$ (a subdiagram of $\cY_j$). Then $\cX_j -\cX_{j-1}=\cX(\underline{n}-U)$ is the limit of the diagram $\partial^{\underline{n}}_{\underline{n}-U} \cX \to \partial^{\underline{n}}_{\underline{n}-U} \cY$. The map $\theta_j:\lim(\cX_j \to \cY_j) \to \lim(\cX_{j-1} \to \cY_j)$ is the pullback of the map from $\cX(\underline{n}-U)$ to the punctured cube $(\partial^{\underline{n}}_{\underline{n}-U} \cX)_p \to \partial^{\underline{n}}_{\underline{n}-U} \cY$, where $\cZ_p$ indicates removing the initial vertex from the cube $\cZ$. Since the latter map is a corner map for a subcube $\cX' \to \cY'$, it is in $\cF$ by hypothesis. Since $\cF$ is preserved by pullbacks, $\theta_j \in \cF$, and thus $L_j \to L_{j-1}$ is also in $\cF$. 
  
  \medskip

   Finally, $L_{2^n-1}=\lim(\cX \to \cY)=\lim(\cX)$, so we have factored the map $\lim \cX \to \lim \cY$ as a composition of maps of $\cF$
   \[
   \lim \cX =L_{2^n-1} \to \cdots \to L_1 \to L_0 \to \lim \cY.
   \] \end{proof}

We note that if $\cX$ and $\cY$ are punctured cubes (with initial vertex missing), we can complete them into cubical diagrams by setting 
$$\cX (\emptyset) = \lim_{U \neq \emptyset} \cX (U)$$
and similarly for $\cY$. Then the corner map of the whole cube map $\cX \to \cY$ is an isomorphism, and is thus in $\cF$.

\medskip
We will now continue to prove results about the classes of maps $\mathcal{I}$, $\mathcal{J}$, and $\mathcal{W}$, and the extension and lifting properties that they satisfy.  Recall that $\icell$ is the collection of morphisms obtained by transfinite composition of pushouts of coproducts of elements of $\mathcal{I}$, $\iinj$ is the collection of morphisms with the right lifting property with respect to elements of $\mathcal{I}$, and $\icof$ is the collection of morphisms with the left lifting property with respect to maps in $\iinj$.  

\begin{lemma}\label{lem-I-inj}
The classes $\mathcal{I}$, $\mathcal{J}$, $\mathcal{W}$ satisfy $\mathcal{W} \cap \jinj \subseteq \iinj$.
\end{lemma}

\begin{proof}
    Let $f:X \to Y$ be in $\mathcal{W} \cap \jinj$. We will show $f \in \iinj$, that is, $f$ has the right lifting property with respect to $s_n \Box b^\mathbf{H}$ for all $\mathbf{H}$ and all $n$ by inducting on the length of $\mathbf{H}$. 

\medskip

Let $\mathbf{H}=\{ H_0 \}$ be a chain of length 0. By part (2) of Lemma \ref{lem-eltry-cof}, the generating acyclic cofibrations in the elementary model structure are in $\jcell$. Thus $\jinj \subseteq \cJ^{elem}$-$\mathsf{inj}$, so $f\in \cW \cap \jinj \subseteq \cW \cap \cJ^{elem}$-$\mathsf{inj}=\cI^{elem}$-$\mathsf{inj}$. Thus $f$ has the right lifting property with respect to $s_n \times \Delta^{H_0} =s_n \Box b^{H_0}$. 

\medskip

Suppose for induction that for all $\mathbf{K}=K_0 < \cdots < K_\ell$ of length less than $m$, $f$ has the right lifting property with respect to $s_n \Box b^\mathbf{K}$. Suppose $\mathbf{H}=H_0< \cdots< H_m$ has length $m$. 

\medskip

In the adjunction of pushout-product with pullback-hom from Lemma \ref{lemma:popr-adj-pbhom}, let $\cC=\topp$ and $\sD=\cE=\itop$, let $\otimes$ be the cartesian product and $\Hom_\sD=\imap$ . We will show $f$ has the right lifting property with respect to $s_n \Box b^\mathbf{H}$ by showing that the map of spaces 
\[
\imap_{\Box}(b^\mathbf{H},f): \imap(\Delta^\mathbf{H},X) \to \imap(\Delta^\mathbf{H},Y) \times_{\imap(\partial \Delta^\mathbf{H}, Y)} \imap(\partial \Delta^\mathbf{H}, X)
\]
 is in $\mathcal{I}^\topp$-$\mathsf{inj}$ $=\mathcal{W}^\topp \cap \mathcal{J}^\topp$-$\mathsf{inj}$. 

\medskip

 Since $f$ has the right lifting property with respect to $s_n \Box b^\mathbf{J} \Box i_0$ for all $\mathbf{J}$ (of any length) and all $n$, the map $\imap_{\Box}(b^\mathbf{H},f)$ has the right lifting property with respect to $s_n \Box i_0 \cong D^{n+1} \times i_0 \in \mathcal{J}^{\topp}$. That is, $\imap_{\Box}(b^\mathbf{H},f)$ is in $\cJ^{\topp}$-$\mathsf{inj}$. We need only show that $\imap_{\Box}(b^{\mathbf{H}},f)\in \cW^\topp$.
 
 \medskip

Since the map $\imap(\Delta^\mathbf{H},f) \in \mathcal{W}$ is factored by $\imap_{\Box}(b^\mathbf{H},f)$, the 2-out-of-3 property for weak equivalences implies $\imap_{\Box}(b^\mathbf{H},f)$ is a weak equivalence in $\topp$ if the left vertical map of the following pullback is a weak equivalence in $\topp$.
\[
\xymatrix{
  \imap(\Delta^\mathbf{H},Y) \times_{\imap(\partial \Delta^\mathbf{H}, Y)} \imap(\partial \Delta^\mathbf{H}, X) \ar[r] \ar[d] & \imap(\partial \Delta^\mathbf{H}, X) \ar[d]^-{\imap(\partial \Delta^\mathbf{H}, f)} \\
  \imap(\Delta^\mathbf{H},Y) \ar[r] & \imap(\partial \Delta^\mathbf{H}, Y)
}
\]

It suffices to show that the right vertical map $\imap(\partial \Delta^\mathbf{H}, f)$ is an acyclic fibration of topological spaces (that is, $\imap(\partial \Delta^\mathbf{H}, f) \in \cI^\topp$-$\mathsf{inj}$). We will use Lemma \ref{cubelimitsfibration} and the fact that 
$\partial \Delta^\mathbf{H}=colim_{\bullet<\mathbf{H}} \Delta^\bullet$. 

\medskip

 Let $\cF=\cI^\topp$-$\mathsf{inj}$, which is closed under pullbacks and composition and contains isomorphisms. 
Let $\cX$ be the punctured cube defined by $\{\mathbf{K} \mapsto \imap(\Delta^{\mathbf{K}}, X)\}$ and let $\cY$ be the punctured cube defined by $\{\mathbf{K} \mapsto \imap(\Delta^{\mathbf{K}}, Y)\}$ where $\mathbf{K}$ ranges over the proper subchains of $\mathbf{H}$ (including the empty subchain). The map of cubical diagrams $\cX \to \cY$ is given by $\imap(\Delta^{\mathbf{K}}, f)$. Because each $\imap_{\Box}(b^\mathbf{K},f) \in \cI^\topp$-$\mathsf{inj}$ by the induction hypothesis, all corner maps of subcube maps $\cX' \to \cY'$ are in $\cF$. Then by the lemma, $\lim \cX \to \lim \cY$ is an acyclic fibration of spaces, where $\lim \cX=\lim_{\bullet<\mathbf{H}} \imap(\Delta^{\bullet},X)=\imap(\partial\Delta^{\mathbf{H}},X)$. We have shown that $\imap(\partial\Delta^\mathbf{H},f)$ is an acyclic fibration of spaces, so $\imap_{\Box}(b^\mathbf{H},f)$ is an acyclic fibration of spaces, ensuring $f$ has the right lifting property with respect to $s_n \Box b^\mathbf{H}$. 
\end{proof}

\subsection{Manifolds are built out of isovariant cells}

We will now show that manifolds are built out of maps in $\mathcal{I}$. That is, we will show that $\emptyset \to M \in \icell$ for all smooth $G$-manifolds $M$. This is not true in the elementary model structure.

\begin{ex}
Consider the disk $D^2$ with the $C_2$-action which rotates the disk around the origin by $\pi$. We can build this as a $C_2$-isovariant cell complex with two mixed cells and one free cell: $D^0 \times \Delta^{e<G}$, $D^1 \times \Delta^e$, and $D^1 \times \Delta^{e<G}$, where the free $1$-cell is glued to $x$ and $\tau x$ and the mixed $1$-cell is glued by collapsing the entire $D^1 \times \Delta^G$ to $z$ and sending $\{0\} \times fd (\Delta^{e<G})$ to $a$ and $\{1\} \times fd(\Delta^{e<G})$ to $\tau a$. 

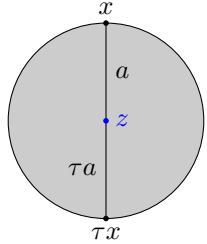
\begin{center}
\begin{tikzpicture}[scale=.65]
  \filldraw[fill=black!20] (0,0) circle (2cm);
   \draw (0,2)--(0,-2);
  \filldraw[blue] (0,0) circle (1.3pt);
  \filldraw (0,2) circle (1.3pt);
  \filldraw (0,-2) circle (1.3pt);
  \node[below] at (0,-2) {$\tau x$};
  \node[above] at (0,2) {$x$};
  \node[right] at (0,1) {$a$};
  \node[left] at (0,-1) {$\tau a$};
  \node[right, blue] at (0,0) {$z$};
\end{tikzpicture}
\end{center}
\end{ex}

This is an example of a $C_2$-space for which $\emptyset \to D^2 \in \icell$, but it is not cofibrant in the elementary model structure on $\itop$, so $\emptyset \to D^2 \not \in \mathcal{I}^{elem}$-$\mathsf{cell}$.

\medskip

To prove that $\emptyset \to M \in \icell$ , we show that an equivariant triangulation gives an isovariant cell structure. In \cite{illmanCW}, Illman proves that smooth $G$-manifolds are cofibrant in $\etop$. We build on his proof in this section. Illman defines the equivariant simplex $\Delta_n(G;H_0, \dots, H_n)$ the same way we define the linking simplex $\Delta^{H_n< \cdots < H_0}_G$, but he allows redundancies in the subgroups. Illman's key result is Theorem 4.1, which proves that if $X$ is a $G$-space with a homeomorphism $u:\Delta^n \to X/G$ which has constant isotropy type in each of the sets $\Delta^m -\Delta^{m-1}$ for $0\leq m \leq n$, then there exist closed subgroups $H_0\geq  \cdots \geq H_n$ of $G$ and a $G$-homeomorphism $\alpha: \Delta_n(G; H_0, \dots, H_n) \to X$ which induces $u$ on orbits. 
  
  \medskip
  
  An \emph{equivariant triangulation} on a $G$-space $X$ consists of a triangulation $t:K \to X/G$ of the orbit space such that for each $n$-simplex $s$ of the simplicial complex $K$, there are closed subgroups $H_0, \dots, H_n$ of $G$ and a $G$-homeomorphism $\alpha:\Delta_n(G;H_0, \dots, H_n) \to \pi^{-1}(t(s))$ which induces a linear homeomorphism on orbit spaces (\cite[5.1]{illmanCW}). Illman shows in Proposition 6.1 that if a $G$-space $X$ has an equivariant triangulation, then it has the structure of a $G$-CW complex. That is, $X$ is cofibrant in $\etop$ and can be built from the equivariant generating cofibrations. 
  Illman cites \cite{veronastrat} to prove that if $M$ is a smooth $G$-manifold, then $M$ has an equivariant triangulation. 

\begin{thm}\label{thm-mflds-cofibrant}
  Let $G$ be a finite group. Smooth $G$-manifolds $M$ satisfy $\emptyset \to M \in \icell$.
 \end{thm} 

 \begin{proof}
  We will prove the more general fact that a $G$-space $X$ with an equivariant triangulation satisfies $\emptyset \to X \in \icell$. Let $X$ be a $G$-space with an equivariant triangulation $t:K \to X/G$. Consider the $n$-skeleton of $K$, denoted $K^n$. 
  The $n$-skeleton of the isovariant space $X$ is $\pi^{-1}(t(K^n))$, where $\pi:X \to X/G$ is the quotient to the orbit space. 
  
  \medskip

  The goal is to construct an isovariant characteristic map for each equivariant simplex $\pi^{-1}(t(s))$ of $X$. Let $s$ be an arbitrary $n$-simplex of $K$, and let $\alpha: \Delta_n(G;H_0, \dots, H_n) \to \pi^{-1}(t(s))$ be the $G$-homeomorphism arising from the definition of equivariant triangulation. The map $\alpha$ is also an isovariant homeomorphism  because it is an injective equivariant map with injective equivariant inverse. Then $\alpha^\ast:\Delta^n \to t(s)$ is a linear homeomorphism with $\alpha^\ast(\partial {\Delta}^n)=t(\partial {s})$ and $\alpha^\ast(\mathring{\Delta}^n)=t(\mathring{s})$, where $\partial {s}$ is the boundary of $s$ and $\mathring{s}$ is its interior. In the isovariant category, we have
  \[
\xymatrix{
  \partial{\Delta}_n(G;H_0, \dots, H_n) \ar[r]^(.6)\cong \ar[d] & \pi^{-1}(t(\partial s)) \ar[d] \\
  \Delta_n(G;H_0, \dots, H_n) \ar[r]^(.6)\cong & \pi^{-1}(t(s))
}
  \]

  We will show that the left vertical map is a pushout along some $s_m \Box b^\mathbf{K} \in \cI$. We will define an isovariant map $\phi: (\prod \Delta^{n_i}) \times \Delta^\mathbf{K} \to \Delta_n(G;\mathbf{H})$ by defining $\phi$ on fundamental domains of $\Delta_n(G; \mathbf{H})$ and $\Delta^\mathbf{K}$, then extending equivariantly. The map $\phi$ will restrict to an appropriate map on boundaries which produces the desired pushout, because $\prod \Delta^{n_i} \cong D^m$.
  
\medskip

  Let $\mathbf{H}=H_0\geq \cdots \geq H_n$ be the chain of isotropy groups of $\Delta_n(G;H_0, \dots, H_n)$, and let $\mathbf{K}=K_0> K_1> \cdots > K_k$ be the same chain but with no degeneracies, that is if $H_i=H_{i+1}$ in $\mathbf{H}$, $H_i$ is only listed once in $\mathbf{K}$. Let $p: \mathbf{n} \to \mathbf{k}$ be the ordered surjection recording the collapse from $\mathbf{H}$ to $\mathbf{K}$, that is, $H_j$ is the same for each $j \in p^{-1}(i)$ for all $0 \leq i \leq k$. (Here $\mathbf{n}=\{0,1, \dots, n\}$.)
  
  \medskip

  Let $v_{K_i}$ be the vertex of a fundamental domain of $\Delta^\mathbf{K}$ which has isotropy group $K_i$. Denote the vertices of $\Delta^{n_i}$ by $v_0, ..., v_{n_i}$, and for each $0 \leq i \leq k$, let $0 \leq \ell_i \leq |p^{-1}(i)|-1$. In a fundamental domain of $\Delta_n(G;\mathbf{H})$, denote the $i$th vertex by $w_i$.
  
  \medskip

Let $\phi$ be the map 
\[
\Delta^{|p^{-1}(0)|-1} \times \cdots \times \Delta^{|p^{-1}(k)|-1} \times fd(\Delta^{\mathbf{K}}) \to fd(\Delta_n(G; \mathbf{H}))
\]

defined on vertices by 
\[
(v_{\ell_0}, v_{\ell_1}, \dots, v_{\ell_k}, v_{K_i})\mapsto w_{min\{p^{-1}(i)\}+\ell_i}.
\]

This map projects to $v_0$ in all trivial (with respect to the $G$-action) simplices except the $i$th, leaving a simplex of dimension $|p^{-1}(i)|-1$ with isotropy $K_i$, which is isovariantly homeomorphic to the corresponding face with isotropy $K_i$ in $\Delta_n(G; \mathbf{H})$. That is, $\phi$ sends
\[
\xymatrix{\Delta^{|p^{-1}(0)|-1} \times \cdots \times \Delta^{|p^{-1}(k)|-1} \times \{v_{K_i}\} \ar[r]^(.63){proj} & \Delta^{|p^{-1}(i)|-1} \times \{v_{K_i}\} \ar[r]^(.545)\cong &  \text{span}(w_{p^{-1}(i)}).}
\]

Here $w_{p^{-1}(i)}$ denotes the subset of vertices $\{ w_j: j \in p^{-1}(i) \}$ of $\Delta_n(G;\mathbf{H})$.

\medskip

Let $n_i=|p^{-1}(i)|-1$, and let $b^{n_i}: \partial \Delta^{n_i} \to \Delta^{n_i}$ denote the boundary inclusion. Note that $\prod_{i=0}^k \Delta^{n_i} \cong D^{n-k}$, and the boundary (whose inclusion is defined as a pushout product, that is, the map $b^{n_0} \Box \cdots \Box b^{n_k}: \partial (\prod \Delta^{n_i}) \to \prod \Delta^{n_i}$) is homeomorphic to $S^{n-k-1}$. Therefore the map $b^{n-k} \Box b^\mathbf{K}$ is isovariantly homeomorphic to the map $s_{n-k-1} \Box b^\mathbf{K} \in \mathcal{I}$. 

\medskip

Note that $\phi$ restricts to an isovariant map 
\[
\phi: \partial (\prod_{i=0}^k \Delta^{n_i}) \times \Delta^{\mathbf{K}} \coprod_{\partial (\prod_{i=0}^k\Delta^{n_i}) \times \partial \Delta^{\mathbf{K}}} (\prod_{i=0}^k\Delta^{n_i}) \times \partial \Delta^{\mathbf{K}} \to \partial \Delta_n(G; \mathbf{H}).
\]

  Careful consideration shows that $\Delta_n(G;H_0, \dots, H_n)$ is the pushout 
  \[
\xymatrix{
  \partial (\prod_{i=0}^k\Delta^{n_i}) \times \Delta^{\mathbf{K}} \coprod_{\partial (\prod_{i=0}^k\Delta^{n_i}) \times \partial \Delta^{\mathbf{K}}} (\prod_{i=0}^k\Delta^{n_i}) \times \partial \Delta^{\mathbf{K}} \ar[d]^{b^{n-k} \Box b^{\mathbf{K}}} \ar[r] & \partial {\Delta}_n(G;H_0, \dots, H_n) \ar[d] \\
  (\prod_{i=0}^k\Delta^{n_i}) \times \Delta^{\mathbf{K}} \ar[r]^{\phi} & \Delta_n(G;H_0, \dots, H_n)
}
  \]

  Since all smooth $G$-manifolds have an equivariant triangulation (\cite[3.8]{veronastrat}), the proof is complete.
\end{proof}

In section \ref{sec-isvt-fixed-pt}, we will need the fact that some subspaces of a $G$-manifold $M$ have inclusions which are in $\icell$. We record the result here. 

\begin{lemma} \label{subcomplex}
  The inclusion of an equivariant subsimplex into a linking simplex is in $\icell$.
\end{lemma}

\begin{proof}
Let $\mathbf{H}= H_0< \cdots < H_k$ and let $A$ be an equivariant subsimplex of the linking simplex $\Delta^\mathbf{H}$. That is, $A=\Delta^{\widehat{\mathbf{H}}}$, where $\widehat{\mathbf{H}}$ is a subchain of $\mathbf{H}$. Let $H_{i_1}, \dots, H_{i_m}$ denote the subgroups of $\mathbf{H}$ which are not in $\widehat{\mathbf{H}}$.

We will build $\Delta^\mathbf{H}$ from $A$ as a composition of pushouts along maps of $\mathcal{I}$. When $n=0$, the map $s_{n-1} \Box b^\mathbf{K}$ of $\mathcal{I}$ is the map $b^\mathbf{K}$ for any $\mathbf{K}$. Then if $A_0=A \coprod \coprod_{j=1}^m \Delta^{H_{i_j}}$ is the union of $A$ with the orbits of the missing subgroups $H_{i_j}$, the map $A \to A_0$ is in $\icell$ because it is a pushout along $\coprod b^{H_{i_j}}$. 
Now the length 1 subchains of $\mathbf{H}$ containing $H_{i_j}$ have both length 0 subchains represented in $A_0$; that is, we may glue in the equivariant 1-simplices $\Delta^{H_{i_j}<H'}$ along $\partial \Delta^{H_{i_j}<H'}$.
Let $A_1$ be the pushout of $A_0$ along the coproduct of the generating cofibrations $b^{H_{i_j}<H'}$ and $b^{H' < H_{i_j}}$ for $j=1,\dots, m$. By induction, one can continue in this way until the entire boundary of the linking simplex has been glued in to $A$ along generating cofibrations of the form $b^{\mathbf{K}}$. Finally, $\Delta^\mathbf{H}$ can be obtained by pushout along $b^{\mathbf{H}}$, so $A \to \Delta^{\mathbf{H}}$ is a composition of pushouts of maps in $\mathcal{I}$, so is in $\icell$. 
\end{proof}

\begin{cor} \label{Mk-to-M-is-cofib}
For all $k$, the inclusion $M_{k} \to M$ is in $\icell$.
\end{cor}

\begin{proof}
The equivariant triangulation on $M$ admits $M_{k}$ as an equivariant subtriangulation.
By Lemma \ref{subcomplex}, $M_k \to M$ is in $\icell$.
\end{proof}

\subsection{Manifolds satisfy isovariant lifting properties}
We will now show that smooth $G$-manifolds satisfy isovariant lifting properties with respect to maps in $\jcell$. Recall that the pushout-product map $s_{n-1} \Box i_0$ is homeomorphic to the map $D^{n} \times i_0$ in $\topp$, in the sense that $D^n \times i_0$ is $s_{n-1} \Box i_0$ composed with a homeomorphism. Then for a $G$-space $X$, there is an isovariant lift in the following diagram for any map $j:A \to B$ of $\jcell$ 
\[
\xymatrix{
A \ar[r] \ar[d]_j & X \\
B \ar@{-->}[ur]
}
\]

if and only if there is an isovariant lift in the following diagram for all $n$ and for all chains of subgroups $\mathbf{H}$:

\[
\xymatrix{
I \times D^n \times \partial \Delta^\mathbf{H} \cup_{\{0\}\times D^n \times \partial \Delta^\mathbf{H}} \{0\} \times D^n \times \Delta^\mathbf{H} \ar[rr] \ar[d]_{(i_0 \times D^n) \Box b^\mathbf{H}} && X \\
I \times D^n \times \Delta^\mathbf{H} \ar@{-->}[urr] }
\]

Using the adjunction of the pushout-product with the pullback-hom (Lemma \ref{lemma:popr-adj-pbhom} with $\mathcal{C}=\topp$ and $\sD=\mathcal{E}=\itop$), such an extension corresponds to a lift in the following diagram of spaces, or equivalently the condition that $\imap(b^\mathbf{H}, X)$ is a Serre fibration in $\topp$, since $s_{n-1}\Box i_0$ are the generating acyclic cofibrations of $\topp$. 
\[
\xymatrix{
  \{0\} \times D^n \ar[r] \ar[d] & \imap(\Delta^\mathbf{H}, X) \ar[d] \\
  I \times D^n \ar[r] \ar@{-->}[ur] & \imap(\partial \Delta^\mathbf{H},X)
}
\]

We will show that if $X$ is a smooth $G$-manifold, such a lift exists. 
We need to set up some notation for the proof. 
For $0<\epsilon<1$, let $\Delta^n_{\epsilon}$ be $\{(t_0, \dots, t_n): \sum t_i =1, 1-\epsilon \leq t_0 \leq 1\}$, as pictured in Figure \ref{epsimplicesA}. The corresponding equivariant simplex $\Delta^{\mathbf{H}}_\epsilon$ is the appropriate quotient of $G \times \Delta^n_\epsilon$. This truncated simplex contains the 0th vertex, which is fixed by $H_n$ in a fundamental domain of the equivariant simplex. 

\medskip

Let $\Delta^n_{=\epsilon}$ be the face of $\Delta^n_\epsilon$ which does not contain the 0 vertex and let $\Delta^n_{<\epsilon}$ be its complement in $\Delta^n_{\epsilon}$. We note that $\Delta^n_{=\epsilon} \cong \Delta^{n-1}$ and for the corresponding equivariant simplices, $\Delta^{H_0< \cdots<H_n}_{=\epsilon}$ is isovariantly homeomorphic to $\Delta^{H_0< \cdots<H_{n-1}}$. Finally, for $0<\delta<\epsilon<1$, let $\Delta_{[\delta,\epsilon]}$ denote $\Delta_{\epsilon}\setminus \Delta_{<\delta}$, (pictured in figure \ref{epsimplicesB}). 
Then the boundary is $\partial \Delta_{[\delta,\epsilon]}= (\partial \Delta_{\epsilon} \setminus \partial\Delta_{\delta}) \cup \Delta_{=\delta}$. Denote the inclusion of this boundary by $b$. For the corresponding equivariant simplices, isovariantly $\Delta_{[\delta,\epsilon]}\cong \Delta_{=\epsilon} \times I$ and the inclusion of the boundary is isovariantly homeomorphic to the pushout-product $b\cong (\partial I \to I) \Box (\partial\Delta_{=\epsilon} \to \Delta_{=\epsilon})$ or equivalently, $b \cong s_0 \Box b^{H_0< \cdots < H_{n-1}}$.
We denote by $\delta_0$ the map $\partial \Delta^{n}_\epsilon \setminus \mathring{\Delta}_{=\epsilon} \to \Delta_\epsilon^n$ which includes the boundary components containing the 0th vertex, pictured in figure \ref{epsimplicesC}.

\begin{figure} 
\begin{subfigure}[t]{0.4\textwidth}
\begin{tikzpicture}
  \filldraw[fill=black!20] (0,2)--(4,2)--(2,-1)--(0,2);
  \filldraw[thick, magenta, fill=magenta!30] (1,.5)--(3,.5)--(2,-1)--(1,.5);
  \node[below] at (2,-1) {$v_0$};
  \node[left, magenta] at (1, .5) {$\Delta_{=\epsilon}^n -$};
  \draw [decorate,decoration={brace,amplitude=10pt}, magenta] (3.1,.4) -- (3.1,-1) node [midway,xshift=20pt] {$\Delta_{<\epsilon}^n$};
\end{tikzpicture}
\caption{$\Delta_\epsilon^n \subset \Delta^n$ with bold boundary}
\label{epsimplicesA}
\end{subfigure}
\, \, \, \, 
\begin{subfigure}[t]{0.4\textwidth} 
\begin{tikzpicture}
  \filldraw[fill=black!20] (0,2)--(4,2)--(2,-1)--(0,2);
  \node[below] at (2,-1) {$v_0$};
  \filldraw[thick, blue, fill=blue!30] (.333,1.5)--(3.667,1.5)--(2.667,0)--(1.333,0)--(.333,1.5);
  \node[right, blue] at (4, 1.5) {$- \Delta_{=\epsilon}^n$};
  \node[right, blue] at (4, 0) {$- \Delta_{=\delta}^n$};
\end{tikzpicture}
\caption{$\Delta_{[\delta,\epsilon]}^n \subset \Delta^n$ with bold boundary}
\label{epsimplicesB}
\end{subfigure} 
\vspace{.2in}

\begin{subfigure}[t]{0.4\textwidth} 
\begin{tikzpicture}
  \draw[thick, magenta] (3,.5)--(2,-1)--(1,.5); 
  \draw[->] (3.5,0)--(4.5,0);
  \filldraw[thick, magenta, fill=magenta!30] (5,.5)--(7,.5)--(6,-1)--(5,.5);
   \node[above] at (4,0) {$\delta_0$};
  \node[below] at (2,-1) {$v_0$};
  \node[below] at (6,-1) {$v_0$};
\end{tikzpicture}
\caption{the map $\delta_0: \partial\Delta_\epsilon \setminus \mathring{\Delta}_{=\epsilon} \to \Delta_\epsilon$}
\label{epsimplicesC}
\end{subfigure}
\caption{}
\end{figure}

\medskip

We denote by $\horseshoe$ the domain of a map $b^\mathbf{H} \Box i_0 \in \mathcal{J}$. Let $\corner{\ep}{\ep'} \subset \horseshoe$ be the domain of the pushout-product of $i_0:\{0\} \to [0,\epsilon']$ with $\delta_0:\partial\Delta_\epsilon \setminus \mathring{\Delta}_{=\epsilon} \to \Delta_\epsilon$, pictured for $H_0<H_1<H_2$ in figure \ref{horseshoe} on a fundamental domain. 

\begin{figure}

\tdplotsetmaincoords{60}{120}

\begin{tikzpicture}[tdplot_main_coords]

\draw[dotted] (0,0,0)--(2,0,3)--(-2,0,3)--(0,0,0);
\draw[dotted] (0,4,0)--(2,4,3);
\draw[dotted] (0,0,0)--(0,4,0);
\draw[dotted] (0,4,0)--(-2,4,3);

\filldraw[blue, fill=blue!20] (0,0,0)--(1,0,1.5)--(-1,0,1.5)--(0,0,0);
\draw[orange, thick] (0,0,0)--(1,0,1.5);
\node at (0,0,0)(a1) {${\color{magenta}{\bullet}}$};
\node at (1,0,1.5)(v10) {${\color{orange}{\bullet}}$};

\filldraw[blue, fill=blue!20] (0,0,0)--(-1,0,1.5)--(-1,2,1.5)--(0,2,0);

\filldraw[orange, fill=orange!20] (0,0,0)--(1,0,1.5)--(1,2,1.5)--(0,2,0);
\draw[orange, thick] (0,2,0)--(1,2,1.5)--(1,0,1.5);

\node at (0,2,0)(v01) {${\color{magenta}{\bullet}}$};
\draw[magenta,thick] (0,0,0)--(0,2,0);
\node at (-1,0,1.5)(v20) {${\color{blue}{\bullet}}$};

\draw [decorate,decoration={brace,amplitude=10pt}, magenta] (0,2,-0.2) -- (0,0,-0.2) node [midway,yshift=-20pt] {$[0,\ep']$};
\node at (0,-2,.5) {$fd(\Delta^{\mathbf{H}}_\ep)$};

\draw[dotted] (2,0,3)--(2,4,3)--(-2,4,3)--(-2,0,3);

\end{tikzpicture}

\caption{$\corner{\ep}{\ep'} \subset fd(\Delta^{H_0<H_1<H_2})\times I$}
\label{horseshoe}
\end{figure}

\begin{thm}\label{thm-mflds-fibrant}
  For a finite group $G$, a smooth $G$-manifold $M$, and a map $j \in \jcell$, there is an isovariant lift in the following diagram:

\[
\xymatrix{A \ar[r] \ar[d]_j &M \\ B \ar@{-->}[ur]}
\]
\end{thm}

We note that having this lifting condition against maps of $\mathcal{J}$ implies that it holds for maps $j \in \jcell$.

The lifting condition above resembles the fibrancy condition for stratified simplicial sets in \cite{douteaustrat}. Proposition 4.12 of \cite{douteaustrat} shows that conically stratified spaces are fibrant, relying on Theorem A.6.4 of \cite{lurieHA}. Our proof below resembles these arguments, but is simplified by the fact that we only consider smooth manifolds rather than the more general conically stratified spaces.

\begin{proof}

We will show that $M$ has the extension property with respect to $i_0 \Box b^\mathbf{H}$ by induction on the length of the chain $\mathbf{H}=H_0< \cdots < H_n$. The same proof shows $M$ has the extension property with respect to $(i_0 \Box b^\mathbf{H}) \times D^n$, so we omit the disk factor throughout this proof.

\medskip

For $\mathbf{H}=H_0$, a length 0 chain, the lifting condition on $\imap(\Delta^\mathbf{H},M) \to \imap(\partial \Delta^\mathbf{H},M)$ reduces to fibrancy of the space $\imap(G/H_0,M) = M_{H_0}$. 

\medskip

Assume that the maps $i_0 \Box b^{H_0< \cdots <H_{n-1}}$ extend isovariantly against any $G$-manifold $M$ for any length $n-1$ chain of subgroups $H_0< \cdots <H_{n-1}$.

\medskip

We wish to show that the map $i_0 \Box b^{H_0< \cdots <H_{n}}$ extends against a $G$-manifold $M$. 
Denote the domain of $i_0 \Box b^{H_0< \cdots < H_{n}}$ by $\horseshoe$ and let $\phi:\horseshoe \to M$ be the map we wish to extend. By the equivariant tubular neighborhood theorem, there is a neighborhood of $\phi(0 \times \Delta^{H_{n}})$ in $M$ which is $G$-homeomorphic to $G \times_{H_{n}} V$ where $V$ is a $H_n$-representation. Let $\corner{\ep}{\ep'} \subset \horseshoe$ be the domain of the pushout-product of $i_0:\{0\} \to [0,\epsilon']$ with $\delta_0:\partial\Delta_\epsilon \setminus \mathring{\Delta}_{=\epsilon} \to \Delta_\epsilon$ (as pictured in Figure \ref{horseshoe}). By continuity, there exist $\epsilon$ and $\ep'$ such that $\phi(\corner{\ep}{\ep'}) \subset G \times_{H_n} V$.

\medskip

An extension of $\phi$ for one connected component of $\corner{\ep}{\ep'}$ mapping to one copy of $V$ will define the extension $\Phi:[0,\epsilon'] \times \Delta^{H_0< \cdots <H_{n}}_\epsilon \to G \times_{H_n} V$ for all components by equivariance. By compactness, this yields an extension of $\phi$ to $I \times \Delta_\epsilon \to M$; fibrancy of the space $M_{H_0}$ yields the extension to all of $I \times \Delta^{H_0< \cdots <H_n}$.

\medskip

Decompose the $H_n$-representation $V$ as the product $V^{H_n} \times W$ where $W=(V^{H_n})^\perp$ and $W^{H_n}=0$. The map $\phi$ can be written as the product $\phi=\phi^{fix}\times \psi$ with $\phi^{fix}:\corner{\ep}{\ep'} \to V^{H_n}$ and $\psi:\corner{\ep}{\ep'} \to W$. Isovariance of $\phi$ implies $\psi$ is isovariant and $\psi([0,\epsilon] \times \Delta^{H_n})=0\in W$.

\medskip

We will extend $\phi^{fix}$ and $\psi$ individually and define the desired extension $\Phi:[0,\epsilon'] \times \Delta^{H_0<\cdots <H_n}_\epsilon \to V$ by $\Phi=\Phi^{fix} \times \Psi$. If $\Psi$ is isovariant, then $\Phi$ is an isovariant extension of $\phi$.

\medskip

The extension $\Phi^{fix}$ of $\phi^{fix}$ is obtained as
\[
\xymatrix{
  fd(\corner{\ep}{\ep'}) \ar[rr]^{\phi^{fix}} \ar[d]_{i_0 \Box \delta_0} && V^{H_n} \\
  [0,\epsilon'] \times fd(\Delta^{H_0<\cdots <H_n}_\epsilon) \ar@{-->}[urr]_{\Phi^{fix}}
}
\]
The vertical map is the pushout-product of $i_0:\{0\} \to [0,\epsilon']$ with the map $\delta_0: \partial\Delta_\epsilon \setminus \mathring{\Delta}_{=\epsilon} \to \Delta_\epsilon$. Since $V^{H_n}$ is fibrant as a space and $i_0 \Box \delta_0$ is an acyclic cofibration of spaces, the extension exists. Once the extension is defined on the fundamental domains, it can be defined on the whole space by equivariance.

\medskip

It remains to extend $\psi$ isovariantly to $\Psi:[0,\epsilon'] \times \Delta^{H_0<\cdots<H_n}_\epsilon \to W$. 

\medskip

By continuity of $\phi$, we can choose real numbers $p_1> p_2> \cdots > p_k> \cdots$ such that $\psi(\corner{p_k}{\ep'}) \subset B_{1/k}(W)$, 
that is, $\psi$ sends the truncated pushout-product domain to the ball of radius $1/k$ centered at 0 in $W$. 

\medskip

Since $B_{1/k}(W)$ is an $H_n$-manifold, we may apply the inductive hypothesis with $\mathbf{H}=H_0< \cdots <H_{n-1}$ to isovariantly extend $\psi$ to each level $[0,\epsilon'] \times \Delta_{=p_k}$.
\[
\xymatrix{
  0 \times \Delta_{=p_k}^{\mathbf{H}} \ar[rr]^{\psi} \ar[d]^{i_0}  & & B_{1/k}(W)_{H_0}\\
  [0,\epsilon'] \times \Delta_{=p_k}^{\mathbf{H}} \ar@{-->}[urr]_{\psi'_k}
}
\]

A similar diagram extends $\psi: 0 \times \Delta^\mathbf{H}_{=\epsilon} \to W_{H_0}$ to $\psi_0':[0,\epsilon'] \times \Delta^\mathbf{H}_{=\epsilon}\to W_{H_0}$. The result of the steps so far is pictured in Figure \ref{fibproofA}.

\medskip

Extension to each sector $\psi_k:[0, \epsilon'] \times \Delta_{[p_{k+1},p_{k}]} \to B_{1/k}(W)$ is also given by the inductive hypothesis, using that isovariantly $\Delta_{[p_{k+1},p_{k}]}^{H_0< \cdots< H_n} \cong \Delta^{H_0< \cdots < H_{n-1}} \times I$ and the boundary inclusion map $b$ is isovariantly homeomorphic to $s_0 \Box b^{H_0< \cdots< H_{n-1}}$.
\[
\xymatrix{
  0 \times \Delta_{[p_{k+1},p_{k}]} \coprod_{0 \times \partial \Delta_{[p_{k+1},p_{k}]}} [0,\epsilon'] \times \partial \Delta_{[p_{k+1},p_{k}]} \ar[rr]^(.7){\psi \coprod \psi'_k \coprod \psi'_{k+1}} \ar[d]^{i_0 \Box s_0 \Box b^{\mathbf{H}}}  & & B_{1/k}(W)\\
  [0, \epsilon'] \times \Delta_{[p_{k+1},p_{k}]} \ar@{-->}[urr]_{\psi_k}
}
\]

One of these extensions is pictured in Figure \ref{fibproofB}. 
Similarly, one obtains an isovariant extension $\psi_0:[0,\epsilon'] \times \Delta_{[p_1,\epsilon]} \to W$. 

\medskip

The union of the extensions $\psi_k$ yields an isovariant map
\[
\Psi:[0,\epsilon'] \times \Delta_\epsilon^{H_0< \cdots < H_{n}} \to W
\]
with $\Psi([0,\epsilon'] \times \Delta_{p_k}^{H_0<\cdots <H_n}) \subset B_{1/k}(W)$. Since $\psi([0,\epsilon'] \times \Delta^{H_n})=0$, the map $\Psi$ is continuous on its domain. 
\end{proof}

\begin{figure}
\begin{subfigure}[t]{0.4\textwidth}
\begin{tikzpicture}
    \draw[thick, magenta] (0,0)-- (2,0);
    \draw [decorate,decoration={brace,amplitude=10pt}, magenta] (2,-.2) -- (0,-.2) node [midway,yshift=-20pt] {$[0,\ep']$};
    \draw [cyan] (0,1.7)--(2,1.7);
    \draw [cyan] (0,1.5)--(2,1.5);
    \node [left] at (0,1.5) {$p_k$};
    \draw [cyan] (0,1.2)--(2,1.2);
    \node [left] at (0,1.2) {$p_{k+1}$};
    \draw [cyan] (0,.7)--(2,.7);
    \draw [cyan] (0,.5)--(2,.5);
    \draw [cyan] (0,.2)--(2,.2);
    \draw [thick, blue] (0,0)--(0,2);
    \node at (0,0) {${\color{magenta}{\bullet}}$};
    \node at (0,1.5) {$\bullet$};
    \node at (0,1.2) {$\bullet$};
    \draw [decorate, decoration={brace,amplitude=10pt}] (-.7,0)--(-.7,2) node [midway, xshift=-20pt] {$\Delta_\epsilon$};
\end{tikzpicture}
\caption{Domains of extensions $\psi_k'$ shown horizontally}
\label{fibproofA}
\end{subfigure}
\, \, \, \, 
\begin{subfigure}[t]{0.4\textwidth} 
\begin{tikzpicture}
    \draw[thick, magenta] (0,0)-- (2,0);
    \draw [decorate,decoration={brace,amplitude=10pt}, magenta] (2,-.2) -- (0,-.2) node [midway,yshift=-20pt] {$[0,\ep']$};
    \draw [cyan] (0,1.7)--(2,1.7);
    \draw [cyan] (0,1.5)--(2,1.5);
    \node [left] at (0,1.5) {$p_k$};
    \draw [cyan] (0,1.2)--(2,1.2);
    \node [left] at (0,1.2) {$p_{k+1}$};
    \draw [cyan] (0,.7)--(2,.7);
    \draw [cyan] (0,.5)--(2,.5);
    \draw [cyan] (0,.2)--(2,.2);
    \filldraw [cyan, fill=cyan!45] (2,1.2)--(0,1.2)--(0,1.5)--(2,1.5);
    \draw [thick, blue] (0,0)--(0,2);
    \node at (0,0) {${\color{magenta}{\bullet}}$};
    \node at (0,1.5) {$\bullet$};
    \node at (0,1.2) {$\bullet$};
    \draw [decorate, decoration={brace,amplitude=10pt}] (-.7,0)--(-.7,2) node [midway, xshift=-20pt] {$\Delta_\epsilon$};
\end{tikzpicture}
\caption{Domain of extension $\psi_k$ shaded in}
\label{fibproofB}
\end{subfigure} 
\caption{}
\end{figure}

\bigskip

\subsection{An isovariant Whitehead's theorem}

Using the results proven in the previous subsections, we can prove that isovariant weak equivalences between smooth $G$-manifolds are isovariant homotopy equivalences.

\begin{thm} \label{thm-whitehead}
  (Isovariant Whitehead's theorem) Let $G$ be a finite group, let $Y$ and $Z$ be smooth $G$-manifolds, and let $f:Y \to Z$ be an isovariant map. Then $f$ is an isovariant homotopy equivalence if and only if $f$ induces weak equivalences $\imap(\Delta^\mathbf{H},f)$ for all chains of subgroups $\mathbf{H}=H_0< \cdots < H_n$ of $G$.
\end{thm}

\begin{proof} Recall from Remark \ref{rem-homotopy-equiv} that an isovariant homotopy equivalence between $G$-spaces is an isovariant weak equivalence. 
For $G$-spaces $X$ and $Y$, let $\iclass{X}{Y}$ denote the set of isovariant homotopy classes of isovariant maps between them. Suppose that $Y$, $Z$ are smooth $G$-manifolds, and $f: Y \to Z$ is an isovariant weak equivalence. We will prove that if $M$ is a smooth $G$-manifold, then $f$ induces a bijection
$$f_*: \iclass{M}{Y} \to \iclass{M}{Z}.$$
When $M=Z$, the preimage of $id_Z$ gives an isovariant homotopy inverse for $f$.

\medskip

In order to do this, we will first replace $f: Y \to Z$ with an isovariant map $\hat{f}: \hat{Y} \to Z \in \jinj$, where $\hat{Y}$ is isovariantly homotopy equivalent to $Y$. Define 
$$\hat{Y} = \{ (y , \gamma) \in Y \times \map(I,Z) : f(y) = \gamma(0), G_{\gamma(t)} = G_{\gamma(0)} \forall t \in I \}$$
with diagonal $G$-action and
$$\hat{f}(y, \gamma) = \gamma(1).$$
This is very similar to the process of replacing a map of topological spaces by a fibration, except the paths are required to stay in the same isotropy subspace of $Z$ in order for the evaluation at 1 to be isovariant. Denote by $\ipath Z$ the paths in $Z$ that remain in the same isotropy subspace. That is, 
$$\ipath Z = \{ \gamma: I \to Z: G_{\gamma(t)} = G_{\gamma(0)} \forall t \}.$$
Then $\hat{Y} = \{ (y, \gamma) \in Y \times \ipath Z : f(y) = \gamma(0) \}$. We will prove a series of claims that will imply the theorem.

\medskip

\textbf{Claim 1: the space $\hat{Y}$ is isovariantly homotopy equivalent to $Y$.} Define $pr: \hat{Y} \to Y$ by projection onto the first component, and $c: Y \to \hat{Y}$ as $c(y) = (y, const_{f(y)})$. These maps are isovariant due to the requirement that paths in $\hat{Y}$ remain in the same isotropy subspace. In addition, $pr \circ c = id_Y$, so it remains to show that $c \circ pr$ is isovariantly homotopic to $id_{\hat{Y}}$.  For $(y, \gamma) \in \hat{Y}$, $c \circ pr (y, \gamma) = (y, const_{f(y)})$. For $s \in [0,1]$, define 
$H((y, \gamma), s) = (y, \gamma_s)$, where $\gamma_s(t) = \gamma(st)$. Then $H$ is an isovariant homotopy from $c \circ pr$ to $id_{\hat{Y}}$, as required.

\medskip

\textbf{Claim 2: the map $ev_0 \times ev_1 : \ipath Z \to Z \times Z$ is in $\jinj$.} In the notation of Lemma \ref{lemma:popr-adj-pbhom}, take $\mathcal{C} = \mathcal{E} = \itop$, and $\sD = \topp$. Here we use the tensoring and cotensoring of $\itop$ over $\topp$. Note that $ev_0 \times ev_1 = \Hom_\Box (i_{0,1}, t_Z)$, where $i_{0,1}:\{0,1\} \to [0,1]$ denotes the inclusion of the endpoints, and $t_Z$ denotes the map from $Z$ to the terminal object in $\itop$.  By the adjunction between pushout-product and pullback-hom, a lift in the diagram
$$\xymatrix{
A \ar[r] \ar[d]_-{  b^\mathbf{H} \Box s_{n-1} \Box i_0} & \ipath Z \ar[d]^-{ev_0 \times ev_1}\\
B \ar[r] \ar@{-->}[ur] & Z \times Z
}$$ 
is equivalent to a lift in the diagram

\[
\xymatrix{
A' \ar[r] \ar[d]_-{  b^\mathbf{H} \Box s_{n-1} \Box i_0 \Box i_{0,1} } & Z \\
B' \ar@{-->}[ur]
}
\]

Here $A$ and $B$ denote the source and target of $ b^\mathbf{H} \Box s_{n-1} \Box i_0$, and $A'$ and $B'$ denote the source and target of $b^\mathbf{H} \Box s_{n-1} \Box i_0 \Box i_{0,1}$. Since $i_{0,1} \cong s_0$, we have $ s_{n-1} \Box i_{0,1} \cong s_n$, so $b^\mathbf{H} \Box s_{n-1} \Box i_0 \Box i_{0,1} \in \jcell$. By Theorem \ref{thm-mflds-fibrant}, a lift exists in the second diagram, and therefore in the first. Therefore $ev_0 \times ev_1 \in \jinj$.

\medskip

\textbf{Claim 3: the map $\hat{f}$ is in $\jinj$.} Let $A$ and $B$ denote the source and target of $ b^\mathbf{H} \Box s_{n-1} \Box i_0 \in \mathcal{J}$, respectively. Any isovariant map $A \to \hat{Y}$ can be written as $\psi \times \phi$, where $\psi: A \to Y$ and $\phi: A \to \ipath Z$ are isovariant maps with $ev_0 \circ \phi(a) = f \circ \psi(a)$ for all $a \in A$. By Theorem \ref{thm-mflds-fibrant}, there is a lift in the diagram 

\[
\xymatrix{
A \ar[r]^-{\psi} \ar[d]_-{ b^\mathbf{H} \Box s_{n-1} \Box i_0} & Y \\
B \ar@{-->}[ur]_-{\tilde{\psi}}
}
\]
Denote this lift $\tilde{\psi}: B \to Y$. A commutative diagram

$$\xymatrix{
A \ar[r]^-{\psi \times \phi} \ar[d]_-{ b^\mathbf{H} \Box s_{n-1} \Box i_0} & \hat{Y} \ar[d]^-{\hat{f}} \\
B \ar[r]^-{\Phi} & Z
}$$

gives a commutative diagram

$$\xymatrix{
A \ar[r]^-\phi \ar[d]_-{ b^\mathbf{H} \Box s_{n-1} \Box i_0} & \ipath Z \ar[d]^-{ev_1} \\
B \ar[r]^-{\Phi} & Z
}$$

and along with the previously obtained lift $\tilde{\psi}: B \to Y$, we obtain a commutative diagram

$$\xymatrix{
A \ar[r]^-\phi \ar[d]_-{ b^\mathbf{H} \Box s_{n-1} \Box i_0} & \ipath Z \ar[d]^-{ev_0 \times ev_1} \\
B \ar[r]^-{(f \circ \tilde{\psi}) \times \Phi} & Z \times Z
}$$

By Claim 2, there is a lift $\tilde{\Phi}: B \to \ipath Z$ in the diagram above. Defining $F(x) = (\tilde{\psi}(x), \tilde{\Phi}(x))$ for all $x \in B$ then gives a lift $F : B \to \hat{Y}$ in the diagram

$$\xymatrix{
A \ar[r]^-{\psi \times \phi} \ar[d]_-{ b^\mathbf{H} \Box s_{n-1} \Box i_0} & \hat{Y} \ar[d]^-{\hat{f}} \\
B \ar[r]^-{\Phi} & Z
}$$

as required.

\medskip

We will finish the proof of the theorem by showing that for $M$ a smooth $G$-manifold, the isovariant weak equivalence $f$ induces a bijection
$$f_*: \iclass{M}{Y} \to \iclass{M}{Z}.$$

\textbf{Claim 4: $f_*$ is surjective.} Suppose we have an isovariant map $g: M \to Z$.  Note that $f = \hat{f} \circ c$ is an isovariant weak equivalence, and $c$ is an isovariant homotopy equivalence, so by the 2 out of 3 property for isovariant weak equivalences, $\hat{f} \in \mathcal{W}$. By Claim 3, $\hat{f} \in \jinj$. Now, by Lemma \ref{lem-I-inj}, $\mathcal{W} \cap \jinj \subseteq \iinj$, and therefore $\hat{f} \in \iinj$. By Theorem \ref{thm-mflds-cofibrant},  $\emptyset \to M \in \icell$, so there is a lift $\tilde{g}: M \to \hat{Y}$ in the diagram

$$\xymatrix{
\empty & \hat{Y} \ar[d]^{\hat{f}} \\
M \ar[r]^-g & Z
}$$

Composing with $pr: \hat{Y} \to Y$, we obtain by Claim 1 a preimage $pr \circ \tilde{g} : M \to Y$ of $g$.

\medskip

\textbf{Claim 5: $f_*$ is injective.} Suppose that $h_0, h_1: M \to Y$ are isovariant maps satisfying $f \circ h_0 = g_0, f \circ h_1 = g_1$, where $g_0$ and $g_1$ are isovariantly homotopic. Composing with $c: Y \to \hat{Y}$, we obtain a commutative diagram

$$\xymatrix{
M \times \{0,1\} \ar[rr]^-{(c \circ h_0) \amalg (c \circ h_1)} \ar[d] && \hat{Y} \ar[d]^{\hat{f}} \\
M \times I \ar[rr]^-{H'} && Z
}$$
where $H'$ is an isovariant homotopy from $g_0$ to $g_1$, and the left hand vertical map is the inclusion. Since one can find an equivariant triangulation of $M \times I$ with an equivariant subtriangulation of $M \times \{0,1 \}$, by Lemma \ref{subcomplex}, the left vertical map is in $\icell$. Recall that $\hat{f} \in \iinj$. Since $M \times \{0 ,1 \} \to M \in \icell$, a lift $H: M \times I \to \hat{Y}$ exists in the above diagram. Composing with $pr: \hat{Y} \to Y$, we obtain an isovariant homotopy from $h_0$ to $h_1$, as required.

\medskip

This concludes the proof of the theorem.
\end{proof}

In \cite[4.10]{dulaschultz}, the authors prove an isovariant Whitehead theorem for compact smooth $G$-manifolds with treelike isotropy structure.  We show that the isovariant Whitehead theorem above agrees with the theorem of Dula--Schultz in the cases where the latter applies. That is, we will show

\begin{prop}\label{prop-equiv-DS}
Let $X$ and $Y$ be compact smooth $G$-manifolds with treelike isotropy structure, and $f: X \to Y$ an isovariant map which satisfies the hypotheses of Theorem 4.10 of \cite{dulaschultz}. Then $f$ is an isovariant weak equivalence.

\end{prop}

A space has \emph{treelike isotropy structure} if all its isotropy subgroups are normal in $G$ and for each isotropy subgroup $H$, the set of isotropy subgroups $K$ such that $K \subset H$ is linearly ordered by inclusion. For more details, see section 3 of \cite{dulaschultz}.

\medskip

In Theorem 4.10 of \cite{dulaschultz}, the hypotheses on the isovariant map $f$ are:
\begin{itemize}
\item For every isotropy subgroup $H \leq G$, $f^H: X^H \to Y^H$ is a homotopy equivalence
\item For every isotropy subgroup $H \leq G$, $f_H: X_H \to Y_H$ is a homotopy equivalence, and
\item For every isotropy subgroup $H \leq G$, $f$ induces a homotopy equivalence $\partial N^H(X) \to \partial N^H (Y)$, where $N^H(X)$ is a mapping cylinder neighborhood of $Sing(X^H)\subset X^H$ and $Sing(X^H)$ is the set of points in $X^H$ whose isotropy groups are strictly larger than $H$.
\end{itemize}

Let $H'$ be a minimal subgroup of $G$ strictly containing $H$ which appears as an isotropy subgroup in $X$. We denote by $\tub_{X^H}(X^{H'})$ a tubular neighborhood of $X^{H'}$ in $X^H$, and $\partial \tub_{X^H}(X^{H'})$ its boundary.  If $X$ has treelike isotropy structure, the space $\partial N^H(X)$ is a disjoint union of spaces of the form $\partial \tub_{X^H}(X^{H'})$ for such $H'$. This is because, if $H'$ and $H''$  minimally strictly contain $H$ and all three appear as isotropy subgroups, then $X^{H'} \cap X^{H''} = \emptyset$, and we can choose tubular neighborhoods which are also disjoint. We will use the following convenient model for $\partial \tub_{X^H}(X^{H'})$:
$$\partial \tub_{X^H}(X^{H'}) \simeq X_{H < H'} = \{ \gamma: [0,1] \to X^H: \gamma(0) \in X^{H'}, \gamma(t) \in X_H \forall t>0 \}$$

Thus we can rephrase the third Dula--Schultz assumption on $f$ as follows:

\begin{itemize}
\item For every consecutive pair of isotropy subgroups $H < H'$, $f$ induces an equivalence $X_{H < H'} \to Y_{H < H'}$.
\end{itemize}

We will now prove Proposition \ref{prop-equiv-DS}.
 \begin{proof}
Let $X$ and $Y$ be compact smooth $G$-manifolds with treelike isotropy structure, and $f: X \to Y$ an isovariant map satisfying the assumptions above. We will prove that $f$ induces weak equivalences $\imap(\Delta^\mathbf{H},f)$ for all chains of subgroups $\mathbf{H}=H_0< \cdots < H_n$ of $G$ by induction on the length of $\mathbf{H}$. For $n=0$, if $H_0 \leq G$ is an isotropy subgroup, then $\imap(\Delta^{H_0},f)=f_{H_0}: X_{H_0} \to Y_{H_0}$ is an equivalence, by assumption.

\medskip

Now take a chain of subgroups $\mathbf{H}=H_0< \cdots < H_n$, and let $H_0'$ be the smallest isotropy subgroup appearing in $X$ which strictly contains $H_0$ and is contained in $H_1$. Let $\mathbf{H}' = H_1< \cdots < H_n$. We will show that $\imap(\Delta^\mathbf{H},X)$ is naturally weakly equivalent to the homotopy pullback in the diagram

$$\xymatrix{
P \ar[r] \ar[d] & X_{H_0 < H_0 '} \ar[d]^-{ev_0} \\
\imap(\Delta^{\mathbf{H}'},X) \ar[r] & X^{H_0 '} 
}$$

where the bottom map evaluates on any point in the simplex $\Delta^{\mathbf{H}'}$; these maps are all homotopic. The homotopy pullback of the diagram above is equivalent to the homotopy pullback of

$$\xymatrix{
P \ar[r] \ar[d] & \map(\Delta^{n-1}, X_{H_0 < H_0 '}) \ar[d]^-{(ev_0)_\ast} \\
\imap(\Delta^{\mathbf{H}'},X) \ar[r] & \map(\Delta^{n-1}, X^{H_0 '})
}$$

where the bottom map is the inclusion of isovariant maps into all maps. We claim that $ev_0$ is a fibration, and therefore so is $(ev_0)_\ast$. Assuming this, the homotopy pullback is equivalent to the pullback. The pullback in this diagram is the space of maps $[0,1] \times \Delta^{\mathbf{H}'} \to X$ which send $(0, (t_i)) \in [0,1] \times \Delta^{\mathbf{H}'}$ to points with isotropy group the corresponding subgroup in $\mathbf{H}'$, and send $(s, (t_i))$ for $s > 0$ to points with isotropy group $H_0$. Thus the pullback is homeomorphic to the space of isovariant maps, $\imap(\Delta^\mathbf{H},X)$.

\medskip

By the induction hypothesis, $f$ induces weak equivalences $X_{H_0 < H_0 '} \to Y_{H_0 < H_0 '}$, $X^{H_0 '} \to Y^{H_0 '}$, and $\imap(\Delta^{\mathbf{H}'},X) \to \imap(\Delta^{\mathbf{H}'},Y)$. Thus it induces a weak equivalence on the homotopy pullbacks, $\imap(\Delta^\mathbf{H},X) \to \imap(\Delta^\mathbf{H},Y)$. So $\imap(\Delta^\mathbf{H},f)$ is a weak equivalence, as required.

\medskip

It remains to justify the fact that $ev_0: X_{H_0 < H_0 '} \to X^{H_0'}$ is a fibration in $\topp$. In other words, we want to show that there is a lift in every diagram of the form
$$\xymatrix{
D^m \times 0 \ar[r] \ar[d]^-{D^m \times i_0} & X_{H_0 < H_0 '} \ar[d]^-{ev_0} \\
D^m \times I \ar[r] & X^{H_0'}
}$$

Adjointing via the path space description of $X_{H_0 < H_0 '}$, we would like to show that for every map
$$\phi: D^m \times 0 \times I \bigcup_{D^m \times 0 \times 0} D^m \times I \times 0 \to X^{H_0}$$
such that $\phi(v,s,0) \in X^{H_0 '}$ for all $s \in [0,1]$ and $\phi(v,0,t) \in X_{H_0}$ for all $t > 0$, there is an extension
$$\Phi: D^m \times I \times I \to X^{H_0}$$
such that $\Phi(v,s,t) \in X_{H_0}$ for all $t>0$. 

Restricting the $G$-action on $X_{H_0<H_0'}$ to an $H_0'$-action does not change the space because the isotropy of $X$ is treelike. As an $H_0'$-space, $X_{H_0<H_0'}$ is equivalent to the 1-dimensional link $\imap(\Delta^{H_0<H_0'},X)$. Thus the fact that $ev_0 \in \jinj$ follows from the proof of Theorem \ref{thm-mflds-fibrant}.
\end{proof}

This suggests that for isovariant maps between manifolds, it suffices to check weak equivalence on isotropy subspaces and on one-dimensional links. The following result verifies this.

\begin{prop}
Let $f: M \to N$ be an isovariant map between $G$-manifolds. If $\imap(\Delta^H,f)$ and $\imap( \Delta^{H_0 < H_1}, f)$ are weak equivalences for all $H, H_0, H_1 \leq G$, then $f$ is an isovariant weak equivalence.
\end{prop}

\begin{proof}
Notice that $M$ and $N$ are, in particular, stratified spaces, whose strata are the $M_H$ for $H \leq G$. For $\mathbf{J} = H_0 < H_1 < ... < H_n$, denote by $\Delta^J$ a fundamental domain of the isovariant simplex $\Delta^{\mathbf{J}}$. This agrees with the stratified simplex $\Delta^J$ in the sense of \cite[Section 1]{douteaustrat}. Note that 
$$\imap(\Delta^{\mathbf{J}}, M) = \map_{strat} (\Delta^J, M)$$
and similarly for $N$. That is, the space of isovariant maps is equal to a space of stratified maps. Suppose that $\imap(\Delta^{\mathbf{J}}, f)$ is a weak equivalence for all $\mathbf{J}$ of length 0 or 1. Then $\map_{strat} (\Delta^J, f)$ is a weak equivalence for all $J$ of length 0 or 1. The manifolds $M$ and $N$ with their isotropy stratification are conically stratified spaces which are the filtered realization of filtered simplicial sets, so by Theorem 5.4 of \cite{douteaustrat}, this implies that $\map_{strat} (\Delta^J, f)$ is a weak equivalence for all $J$ of any length. Therefore $\imap(\Delta^{\mathbf{J}}, f)$ is a weak equivalence for all $\mathbf{J}$, so $f$ is an isovariant weak equivalence, as required.
\end{proof}

\section{Isovariant fixed point theory} \label{sec-isvt-fixed-pt}
Let $G$ be a finite group, and let $M$ be a compact smooth $G$-manifold such that 
\begin{itemize}
\item $\dim M^H \geq 3$ for all $H$ that appear as isotropy groups in $M$, and
\item for $H < K$ that appear as isotropy groups in $M$, $\dim M^H - \dim M^K \geq 2$.
\end{itemize}
Let $f: M \to M$ be an isovariant map. In this section, we prove that the equivariant Reidemeister trace gives a complete invariant for the isovariant fixed point problem.

\begin{thm}\label{thm-finite-G}
If the fixed points of $f: M \to M$ can be removed equivariantly, then they can be removed isovariantly.
\end{thm}

Write $(H) < (K)$ if $H$ is properly subconjugate to $K$, and let
$$ (H_1) > ... >(H_n) $$
be a total ordering of the subgroups appearing as isotropy groups of $x \in M$, which extends the subconjugacy order. For $1 \leq k \leq n$, denote by $M_k = \bigcup_{i \leq k} M^{(H_i)}$ the subspace of $M$ consisting of all points fixed by a subgroup conjugate to some $H_i$, $i \leq k$. Denote by $M_{(H_i)}$ the points fixed exactly by a subgroup conjugate to $H_i$. We prove Theorem \ref{thm-finite-G} inductively.

\begin{proof}
We prove by induction on $k$ that $f|_{M_k}$ can be isovariantly homotoped to have no fixed points. For $k=1$, the fixed points of $f$ can be eliminated equivariantly, therefore they can be eliminated on $M^{H_1}$. (See, for example, Theorems 5.7 and 10.1 of \cite{MP}.)

\medskip

By induction, assume that $f|_{M_{k-1}}$ has been isovariantly homotoped to a map without fixed points. Thus we have a homotopy
$$M \cup_{M_{k-1} \times 0} (M_{k-1} \times I) \to M$$
Note that the inclusion $M \cup_{M_{k-1} \times 0} (M_{k-1} \times I) \to M \times I$ is the pushout-product of $M_{k-1} \to M$ with $i_0:0 \to [0,1]$. The pushout-product of a map in $\icell$ with $i_0$ is a map in $\jcell$, thus by Corollary \ref{Mk-to-M-is-cofib}, this is a map in $\jcell$. By Theorem \ref{thm-mflds-fibrant}, we can thus extend $M \cup_{M_{k-1} \times 0} (M_{k-1} \times I) \to M$ to an isovariant homotopy $H: M \times I \to M$. The map $f_1 = H(-,1)$ has no fixed points on $M_{k-1}$, and therefore also has no fixed points on a neighborhood of $M_{k-1}$. In particular, $f_1$ has no fixed points on a neighborhood of $M_{k-1}$ in $M_k$. Denote by $U$ a $G$-invariant neighborhood of $M_{k-1}$ in $M_k$, whose closure $\bar{U}$ equivariantly deformation retracts to $M_{k-1}$, and such that $f_1$ has no fixed points on $\bar{U}$
. This can be achieved by, for each $i \leq k-1$, taking $V_i$ a small enough equivariant tubular neighborhood of $M^{(H_i)}$ in $M$, and setting $U = \bigcup_{i \leq k-1} V_i \cap M_k$. We'll extend $f_1|_{\bar{U}}$ isovariantly to a map that has no fixed points on $M_k$; it suffices to extend it from $\partial U$ to $M_k - U$. As all points in $\partial U$ and $M_k - U$ have isotropy groups conjugate to $H_k$, an equivariant extension will provide an isovariant extension. As in \cite{Pon}, \cite{KW2}, or \cite{PonKW}, the obstruction $R_G(f|_{\partial U, M_k -U})$ for extending the homotopy which removes the fixed points of $f_1$ from $\partial U$ to $M_k - U$ lives in
$$[C_M (M_k - U, \partial U), S_M E]^G _M.$$
By excision (e.g., Lemma 7.3.1 of \cite{caryparam}) and by the equivariant deformation retraction between $\bar{U}$ and $M_{k-1}$, 
$$[C_M (M_k - U, \partial U), S_M E]^G _M \cong [C_M (M_k, \bar{U}), S_M E]^G _M \cong [C_M(M_k, M_{k-1}), S_M E]^G _M$$
and under these isomorphisms, $R_G(f|_{\partial U, M_k -U})$ maps to $R_G(f|_{\bar{U}, M_k})$, which maps to $R_G(f|_{M_{k-1} , M_k})$. The Reidemeister trace $R_G(f)$ is trivial, and $M$ satifies the assumptions of Theorem \ref{thm-KW-global-to-local}, and so the local Reidemeister trace $R_G(f|_{M_{k-1} , M_k})$ is trivial. Thus we can extend the homotopy which removes the fixed points of $f_1$ from $\partial U$ to $M_k - U$.


\medskip

Therefore we can extend $f_1$ isovariantly to a map with no fixed points on $M_k$, and we are done.
\end{proof}

\begin{cor}\label{cor-equivt-r}
Let $G$ be a finite group, and $M$ a compact smooth $G$-manifold such that 
\begin{itemize}
\item $\dim M^H \geq 3$ for all $H$ that appear as isotropy groups in $M$, and
\item For $H < K$ that appear as isotropy groups in $M$, $\dim M^H - \dim M^K \geq 2$
\end{itemize}
Let $f: M \to M$ be an isovariant map. Then $f$ is isovariantly homotopic to a map without fixed points if and only if its equivariant Reidemeister trace, $R_G(f)$, is trivial.
\end{cor}

However, if $X$ is a compact $G$-space which is not a manifold and $f: X \to X$ is an isovariant map, $R_G(f)$ is not necessarily a complete invariant for isovariantly removing fixed points of $f$:

\begin{ex}
Let $X = D(sgn) \cup_{\{0\}} S^1$, the disk in the sign representation of $C_2$ attached along its fixed point to a circle with trivial action, pictured below. The identity map of $X$ is isovariant. Equivariantly, this space is equivalent to $S^1$ with the trivial action, so the fixed points of the identity map can be eliminated equivariantly (e.g., by rotation). But note that any (continuous) isovariant map $f: X \to X$ must send the point $0 \in X$ to itself: this point is the limit of points in $X_e = D(sgn)-0$, thus $f(0)$ must be in the closure of $D(sgn)-0$. On the other hand, $0 \in X^{C_2}$, so it must be sent into $X^{C_2}$. Thus $f(0) = 0$, so any isovariant $f: X \to X$ has a fixed point. Therefore the fixed points of the identity cannot be eliminated isovariantly.

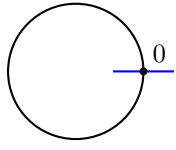
\begin{center}
\begin{tikzpicture}[scale=.9]
    \draw[blue, thick] (1.45,0)--(.55,0);
  \draw[thick] (0,0) circle (1cm);
  \filldraw (1,0) circle (.05cm);
  \node[above right] at (1,0) {$0$};
\end{tikzpicture}
\end{center}

\medskip

We can build a similar example with $\dim X \geq 3$ as follows: let $X = D(sgn^{\oplus 5}) \cup_0 S^3$, where $S^3$ has the trivial action, and take $f : X \to X$ the identity map. Equivariantly, this is equivalent to $S^3$ with the trivial action, so we can remove the fixed points of the identity (e.g. by multiplying by a unit quaternion). The same argument as above shows that the fixed points of the identity cannot be removed isovariantly.

\medskip

Both of these spaces are ``cofibrant", that is, built out of maps in $\icell$. However, they are not ``fibrant"; they do not satisfy the right lifting property with respect to all maps in $\jcell$. A crucial property of smooth $G$-manifolds which makes them fibrant is the existence of tubular neighborhoods of $X^H$ inside larger subspaces $X^K$.
\end{ex}

Possible additional invariants for the isovariant fixed point problem are the equivariant Reidemeister traces of $f$ on $X_{(H)}$ for $H \leq G$, $R_G(f|_{X_{(H)}})$. We do not know whether these are complete invariants; however, when $X$ is not a manifold, the equivariant Reidemeister trace of $f$ on $X_{(H)}$ is not determined by the equivariant Reidemeister trace of $f$. This can be seen by considering the equivariant Lefschetz number of $id|_{X_e}$ in the examples above.

\section{A weak model structure}
In this section, we demonstrate that our results in this paper arise from a weak model structure in the sense of \cite{Hen}. Weak model structures still allow for a good notion of Quillen adjunction (as well as Quillen equivalence), a well-behaved homotopy category, and most other constructions in categorical homotopy theory. We will use Theorem 3.5 of \cite{Hen}, reproduced below.

\begin{thm}
    Let $\mathcal{C}$ be a bicomplete category with two classes of maps $\cI$ and $\cJ$ such that:
    \begin{enumerate}
        \item The classes $\cI$ and $\cJ$ generate weak factorization systems $(\icof, \iinj)$ and $(\jcof, \jinj)$;
        \item $\cJ \subset \icof$; and
        \item $\mathcal{C}$ admits a left adjoint endofunctor $C$ with natural transformations $e_0, e_1: Id \to C$, such that 
        \begin{enumerate}
            \item If $i: A \to B \in \cI$, the map $(B \amalg B)\cup_{A \amalg A} CA \to CB$ is  in $\icof$;
            \item If $i: A \to B \in \cI$, the two maps $B \cup_A CA \to CB$ have the left lifting property with respect to all maps in $\jinj$ between $\cJ$-fibrant objects; and
            \item If $j: A \to B \in \cJ$, then the map $(B \amalg B)\cup_{A \amalg A} CA \to CB$ has the left lifting property with respect to all maps in $\jinj$ between $\cJ$-fibrant objects.
        \end{enumerate}
    \end{enumerate}
    Then $\mathcal{C}$ admits a weak model structure in which $\jinj$ give the fibrations between fibrant objects, and $\icof$ give the cofibrations between cofibrant objects.
\end{thm}

\begin{rem}
    We say that $X \in \mathcal{C}$ is $\mathcal{J}$-fibrant if the map from $X$ to the terminal object is in $\jinj$.
\end{rem}

We will prove that the isovariant category with the classes $\cI$ and $\cJ$ as defined in Definition \ref{defn:geometricij} satisfies the conditions of the theorem above.

\begin{thm} \label{weakmodelstr}
    There is a weak model structure on $\itop$ in which $\jinj$ gives the fibrations between fibrant objects, and $\icof$ gives the cofibrations between cofibrant objects.
\end{thm}

\begin{proof}
    The classes $\cI$ and $\cJ$ indeed generate weak factorization systems by the small object argument because their domains are compact and thus small with respect to relative cell complexes. To prove the second condition, we will show $\iinj \subseteq \jinj$, which implies that $\cJ \subseteq \icof$. Let $f:X \to Y$ be in $\iinj$, so $f$ has right lifting property with respect to $s_n \Box b^\mathbf{H}$. 
By the adjunction of pushout-product with pullback-hom, a lift of $f$ with respect to $s_n \Box b^\mathbf{H}$ is equivalent to a lift in spaces of the map 
\[
\imap_\Box(b^\mathbf{H},f): \imap(\Delta^H, X) \to \imap(\partial \Delta^H, X) \times_{\imap(\partial \Delta^H, Y)} \imap(\Delta^H,Y)
\]
against $s_n \in \cI^{\topp}$. That is, the pullback-hom map above is in $\cI^\topp$-$\mathsf{inj} \subseteq \cJ^\topp$-$\mathsf{inj}$, by the cofibrantly generated model structure on $\topp$. Thus the pullback-hom map lifts in $\topp$ with respect to $s_n \Box i_0\cong D^{n+1} \times i_0$. Using the adjunction again, this implies that $f$ has the right lifting property with respect to $s_n \Box i_0 \Box b^\mathbf{H}$, so $f \in \jinj$. Thus $\iinj \subseteq \jinj$, and therefore $\mathcal{J} \subseteq \icof$.

\medskip

For the third condition, take $CX = X \times I$. This is left adjoint to the functor $\ipath$ from subsection 3.3. We take $e_0: X \to X \times I$ to be the inclusion at 0, and $e_1: X \to X \times I$ to be the inclusion at 1.

\medskip

Now take $i = s_n \Box b^\mathbf{H}: A \to B \in \cI$. Note that the map $(B \amalg B)\cup_{A \amalg A} (A \times I) \to (B \times I)$ is the pushout-product $i_{0,1} \Box i$; since $i_{0,1} \Box s_n \cong s_{n+1}$, this map is in $\icof$, as required. This proves condition 3(a). In addition, the maps $B \cup_A (A \times I) \to B \times I$ are the pushout-products $(0 \to I) \Box i$ and $(1 \to I) \Box i$, and are therefore in $\jcof$, and so satisfy the left lifting property with respect to all maps in $\jinj$. This proves condition 3(b). Finally, if $j = s_n \Box (0 \to I) \Box b^\mathbf{H} \in \cJ$, then the map $(B \amalg B)\cup_{A \amalg A} (A \times I) \to (B \times I)$ is $i_{0,1} \Box j$; since $i_{0,1} \Box s_n \cong s_{n+1}$, this is in $\jcof$, and so satisfies the left lifting property with respect to all maps in $\jinj$. This proves condition 3(c). By Theorem 3.5 of \cite{Hen}, there is a weak model structure on $\itop$ in which $\jinj$ give the fibrations between fibrant objects, and $\icof$ give the cofibrations between cofibrant objects. \end{proof}

\begin{rem}
    Note that our structure is somewhat stronger than that of a weak model category; for instance, the maps in 3(b) and 3(c) of the theorem lift against all maps in $\jinj$, not just the ones between fibrant objects. We conjecture that this weak model structure arises from a semi-model structure (specifically, a left model structure as in \cite{HovSemi}, \cite{Spi}, and \cite{Bar}) or even a model structure (specifically, a fibrantly induced model structure in the sense of \cite{GMSV}.)
\end{rem}

\bibliography{isvtbib}{}
\bibliographystyle{amsalpha}

\end{document}